\documentclass[a4paper,twoside,reqno]{amsart}
\newcommand{\Ueberschrift}{Cubic twin prime polynomials are counted by a modular form}
\newcommand{\Kurztitel}{Cubic twin prime polynomials}
\usepackage{amsmath} \usepackage{amssymb} \usepackage{amsopn}
\usepackage{amsthm} \usepackage{amsfonts} \usepackage{mathrsfs}
\usepackage[headings]{fullpage}
\usepackage[bottom]{footmisc}
\usepackage{mathtools}
\usepackage[dvips]{graphicx} \usepackage[usenames]{color}
\usepackage[all]{xy}
\usepackage[OT2, T1]{fontenc}
\usepackage[utf8]{inputenc}
\usepackage{dsfont}
\usepackage{float}
\usepackage{booktabs}
\usepackage[pdfpagelabels, pdftex]{hyperref}
\hypersetup{
  pdftitle={},
  pdfauthor={},
  pdfsubject={},
  pdfkeywords={},
  colorlinks=true,    
  linkcolor=red,     
  citecolor=blue,     
  filecolor=blue,      
  urlcolor=blue,       
  breaklinks=true,
  bookmarksopen=true,
  bookmarksnumbered=true,
  pdfpagemode=UseOutlines,
  plainpages=false}

\DeclareMathOperator{\rH}{H}

\DeclareMathOperator{\rO}{O}

\DeclareMathOperator{\rc}{c}

\newcommand{\bA}{{\mathbb A}}

\newcommand{\bC}{{\mathbb C}}

\newcommand{\bF}{{\mathbb F}}

\newcommand{\bP}{{\mathbb P}}
\newcommand{\bQ}{{\mathbb Q}}

\newcommand{\bZ}{{\mathbb Z}}

\newcommand{\cE}{{\mathscr E}}

\newcommand{\fS}{{\mathfrak S}}

\newcommand{\fp}{{\mathfrak p}}

\newcommand{\inj}{\hookrightarrow}
\DeclareMathOperator{\id}{id}
\DeclareMathOperator{\pr}{pr}

\DeclareMathOperator{\Aut}{Aut}

\DeclareMathOperator{\GL}{GL}
\newcommand{\tr}{{\rm tr}} 
\newcommand{\matzz}[4]{\left(\begin{array}{cc} #1 & #2 \\ #3 & #4 \end{array} \right)}

\DeclareMathOperator{\Spec}{Spec}
\DeclareMathOperator{\Proj}{Proj}
\newcommand{\redu}{{\rm red}}
\DeclareMathOperator{\Frob}{Frob}

\newcommand{\qrB}[2]{\left\lgroup\frac{#1}{#2}\right\rgroup}
\DeclareMathOperator{\ind}{ind}
\DeclareMathOperator{\Gal}{Gal}

\newcommand{\ph}{\varphi}
\newcommand{\teta}{\vartheta}
\DeclareMathOperator{\Sym}{Sym}

\DeclarePairedDelimiter\abs{\lvert}{\rvert}
\newcommand{\ov}[1]{\mbox{${\overline{#1}}$}} 
\newtheorem{thm}{Theorem}[section]
\newtheorem{prop}[thm]{Proposition}
\newtheorem{lem}[thm]{Lemma}

\newtheorem{cor}[thm]{Corollary}
\newtheorem{conj}[thm]{Conjecture}

\theoremstyle{definition}

\theoremstyle{remark}
\newtheorem{rmk}[thm]{Remark}

\newtheorem{ex}[thm]{Example}

\newenvironment{pro*}[1][Proof]{{\it{#1:}} }{}
\newenvironment{pro**}[1][]{{\it{#1}} }{\hfill $\square$}

\numberwithin{equation}{section}
\begin{document}

\hrule width\hsize
\hrule width\hsize

\vskip 0.5cm

\title[\Kurztitel]{\Ueberschrift} 
\author{Lior Bary-Soroker}
\address{Lior Bary-Soroker, School of Mathematical Sciences, Tel Aviv University, 
Ramat Aviv, Tel Aviv 6997801, Israel}
\email{barylior@post.tau.ac.il}

\author{Jakob Stix}
\address{Jakob Stix, Institut f\"{u}r Mathematik, Goethe--Universit\"{a}t Frankfurt, Robert-Mayer-Stra{\ss}e~{6--8},
60325 Frankfurt am Main, Germany}
\email{stix@math.uni-frankfurt.de}

\thanks{The authors acknowledge support provided by DAAD-Programm 57271540 Strategische Partnerschaften (supported by BMBF). 
The first author 
was partially supported by a grant of the Israel Science Foundation.}

\date{\today} 

\maketitle

\begin{quotation} 
\noindent \small {\bf Abstract} --- 
We present the geometry lying behind counting twin prime polynomials in $\bF_q[T]$ in general. 
We compute cohomology and explicitly count points by means of a twisted Lefschetz trace formula 
applied to these parametrizing varieties for \textit{cubic} twin prime polynomials. 
The elliptic curve $X^3 = Y(Y-1)$ occurs in the geometry, and thus counting cubic twin prime polynomials involves the associated modular form. In theory, this approach can be extended to higher degree twin primes, but the computations become harder. 

The formula we get in degree $3$ is compatible with the Hardy-Littlewood heuristic on average,  agrees with the prediction for $q \equiv 2 \pmod 3$  
but shows anomalies  for $q \equiv 1 \pmod 3$. 
\end{quotation}

\section{Introduction}
One of the most exciting open problems in number theory is the \emph{twin prime conjecture},  predicting 
infinitely many twin primes, i.e., pairs $(n,n+2)$ with both $n$, $n+2$ being primes. There is nothing special about the difference being $2$ or about pairs, and indeed, one may easily state a general conjecture for $k$-tuples of the form $n+h_i$ for a fixed set of shifts $h_1 < h_2 < \ldots <h_k$.
Applications often necessitate a quantitative version of the twin prime conjecture. Hardy and Littlewood  made a precise quantitative conjecture in \cite[\S5 Theorem X{\tiny I}\footnote{Actually not a Theorem but a consequences of Hypothesis X of loc.~cit.}] 
{HardyLittlewood}, which essentially means that the events that `$n$ is prime' and `$n+2$ is prime' are independent up to a precise 
correlation factor 
$\mathfrak{S}(0,2)\approx 1.32032363169\ldots.$

\begin{conj}[The Hardy-Littlewood Prime Tuple Conjecture]\label{conj:HL}
Let $h = (h_1,\ldots, h_k)\in \bZ^k$ be a $k$-tuple of pairwise distinct integers. Then in the limit $x \to \infty$, 
\begin{equation}\label{eq:HLconjecture}
\frac{1}{x}\#\{x\leq n<2x \ : \ n+h_1,\ldots, n+h_k \mbox{ are prime}\}   = \mathfrak{S}(h) \frac{1}{(\log x)^k} (1+o(1)), 
\end{equation}
where, with $\nu_p(h)$ denoting the number of residues covered by $h_1,\ldots,h_k$ modulo $p$, we have  
\begin{equation}\label{eq:HLsing}
\mathfrak{S}(h) = \prod_{p \mbox{ \tiny prime}} \frac{1-\nu_p(h)p^{-1}}{(1-p^{-1})^k}.
\end{equation}
\end{conj} 

Note that a \emph{local obstruction} occurs, if there is  a prime $p$ with $\nu_p(h)=p$. 
Then the conjecture trivially holds, because both sides of \eqref{eq:HLconjecture} are zero.
If there are no local obstructions, 
the singular series $\mathfrak{S}(h)$ may be shown to converge to a positive constant.

In the course of history there were many attempts to solve this conjecture that, in spite of failing, produced exciting results. In his seminal work, Brun \cite{Brun} was the first to apply sieve methods to the twin prime problem and showed that the sum of reciprocals of twin primes converges. Brun also showed that there are infinitely many pairs $(n,n+2)$ with at most $16$ prime factors. Chen \cite{Chen} gave the state-of-the-art result in this direction proving that there exist infinitely many primes $p$ such that $p+2$ is a product of at most $2$ primes. In their famous theorem about arithmetic progression of primes, Green and Tao \cite{GreenTao} and Green, Tao, and Ziegler \cite{GTZ} tackle the easier problem when one more degree of freedom is added.  More recently, Zhang \cite{Zhang} made a breakthrough, by showing that there exist infinitely many prime pairs  with bounded distance. The bound was improved by Polymath8 \cite{Polymath} and later Maynard \cite{Maynard} generalized Zhang's result to $k$-tuples, for any fixed $k>0$. See the survey paper \cite{Granville} for more details. 

\medskip

The objective of this paper is to give more evidence to the Hardy-Littlewood prime tuple conjecture by studying its function field analog. 
For a finite field $\mathbb{F}_q$ 
we take the set 
\[
M_{d,q} = \{ f(T) \in \bF_q[T] \ ; \ \text{ monic of } \deg(f) = d\}
\]
as the analog of the interval $[x,2x)$ with $q^d =\#M_{d,q}$ playing the role of $x\sim\#\{x\leq n<2x\}$. The goal is now to give for a $k$-tuple of pairwise distinct polynomials $h=(h_1,\ldots, h_k)\in \bF_q[T]$ with $\deg h_i<d$, precise asymptotics in the limit $q^d \to \infty$ for 
\[
\pi(d,q;h) = \#\{f \in M_{d,q} \ ; \ f+h_1, \ldots f+h_k \text{ are irreducible}\}.
\]
If we fix $h$, this only makes sense with fixed $q$ and $d \to \infty$, and this case should be considered an analog of Conjecture~\ref{conj:HL} for the field $\bF_q(T)$. If $h = (0,1)$, or if we expect some uniformity in $h$, then we can also ask for the asymptotics for $d$ fixed and $q \to \infty$. 

\medskip

Hall \cite{Hall} proved that for any fixed $q>3$ there exist infinitely many $d$'s and $f\in M_{d,q}$ such that $f$ and $f+1$ are prime. This proof is elementary, and extends to $k$-tuples with scalar shifts if $q$ is sufficiently large with respect to $k$ and using the Riemann hypothesis for curves. The sequence of $d$'s in Hall's proof grows exponentially fast and hence is very sparse. Adapting Maynard's result to function fields, Castillo et al \cite{Castillo} extend Hall's result for sufficiently large fixed $q$, to any $d$ in an explicit arithmetic progression and give a lower bound on the number of pairs of the right order of magnitude. However these method seems to fail when the shift is not a monomial.

Another approach taken in recent studies is to let $q\to \infty$ and control $\pi(d,q;h)$ as an 
application of the Lang-Weil bounds. This direction proved successful as the analogue of the Hardy-Littlewood prime tuple conjecture was completely resolved: Bender and Pollack \cite{Bender}  for pairs in odd characteristic, the first author \cite{BarySoroker} for general tuples in odd characteristic, and Carmon \cite{Carmon} applies the method of the first author in even characteristic, and establishes the following in all cases in the limit $q \to \infty$: 
\begin{equation}\label{HLlargeqlimit}
\pi(d,q;h) =\frac{q^d}{d^k} \mathfrak{S}_q(h)(1+E(d,q;h)),  \qquad \text{ with error term } \quad  E(d,q;h) \ll_{d,k} q^{-1/2}. 
\end{equation}
Note that the size of the error term depends on the tuple $h$ only through its degree bound $d$. 
We discuss the function field analog $\mathfrak{S}_q(h)$ of the Hardy-Littlewood singular series in Section~\S\ref{sec:sing series}. 

One may naively expect a square root cancelation, namely a bound for the error term of the form 
\[
E(d,q;h)  \ll q^{-d/2+\epsilon}.
\] 
Some results give better error terms than \eqref{HLlargeqlimit} on average \cite{KeatingRodittyGershon} and for special $h$ \cite{GorodetskySawin} based on variance computations \cite{KeatingRudnick} and equidistribution results by Katz \cite{Katz1,Katz2}. 
See \cite{HastMatei} for geometric interpretation of \cite{KeatingRudnick}. 
The asymptotic formula \eqref{HLlargeqlimit} has inspired more work, see \cite{ABSR,BBS,BSF,Entin}.
The downside of \eqref{HLlargeqlimit}, is that the error term is not small enough for us to see the `arithmetic' due to the dependence on $h$. Pollack \cite[Appendix]{Pollack} computed that 
\[
\mathfrak{S}_q(h) = 1+O(q^{-1}),
\]
so in the limit $q\to \infty$ the events that the $f+h_i$ are irreducible are indeed independent. 

\medskip

In this paper we give an exact formula for $\pi(d,q;h)$ when $d\leq 3$ and $h = (0,1)$. The case of $d=1$ is trivial and left to the reader.
The case $d=2$ is straight forward using Weil's bounds on exponential sums. We use the technique of the present paper to deduce the $d=2$ case in Example~\ref{ex:caseD=2} as \eqref{eq:caseD=2}:
\[
\pi(2,q;(0,1))  = \frac{1}{4}q^2 \cdot
\begin{cases}
 1- (2 - (-1)^{(q-1)/2} 
 ) \cdot q^{-1}  &  \text{ if $q$ is odd}, \\
 1- 2q^{-1} &  \text{ if $q$ is even}.
\end{cases}
\]

For $d=3$, the formula depends on whether or not the characteristic is $3$. If $3\mid q$ we get a polynomial formula, while if $3\nmid q$, our formula involves the number of rational points in the reduction of the elliptic curve $\cE$ with Weierstra\ss -equation 
\[
X^3 = Y(Y-1).
\]
Recall that if $3\nmid q$, the curve $\cE$ has good reduction and by the Hasse bound
\begin{equation}\label{eq:HB}
\abs{q+1-\#\cE(\bF_q) } \leq 2\sqrt{q}.
\end{equation}

\begin{thm}
\label{thm:main}
Let $q$ be a prime power.
Then, we have
\begin{equation}\label{eq:MTmod3version}
\pi(3,q;(0,1)) = \frac{1}{9} q^3 \cdot
\begin{cases}
1 - q^{-1} - 3q^{-2} & \mbox{if } q \equiv 0 \pmod 3,\\[1ex]
1 + \left(c^2_q - 3\right) \cdot q^{-1} -2q^{-2} &  \mbox{if } q \equiv 1 \pmod 3,\\[1ex]
1 - q^{-1} - 2q^{-2} & \mbox{if } q \equiv 2 \pmod 3,
\end{cases}
\end{equation}
where $c_q := \frac{1+ q - \#\cE(\bF_q))}{\sqrt{q}}$,
so that $0 \leq c^2_q\leq 4$ by \eqref{eq:HB}.
\end{thm}

We will discuss the comparison with the prediction by the analog of the Hardy-Littlewood conjecture in Section~\S\ref{sec:sing series}. 

\begin{rmk}
Since all elliptic curves over $\bQ$ are modular, there is a newform  $f_\cE(q)$ associated to the elliptic curve $\cE$ with Weierstra\ss -equation $X^3 = Y(Y-1)$. The newform can easily be determined being the only newform of weight 2 and level $27$, see the database 
\cite[\href{http://www.lmfdb.org/EllipticCurve/Q/27/a/4}{Elliptic Curve 27.a4}]{LMFDB},  
\[
f_\cE(q) = \sum_{n \geq 1} a_n q^n = q - 2q^{4} - q^{7} + 5q^{13} + 4q^{16} - 7q^{19} + O(q^{20}).
\]
This means that for all prime numbers $p$ not divisible by $3$ we have by the Eichler-Shimura congruence relation 
\[
a_p = 1+p - \# \cE(\bF_p) = \alpha_p + \beta_p,
\]
where $\alpha_p$, $\beta_p$ are the eigenvalues of Frobenius, i.e., the roots of 
\[
T^2 - a_p T + p = 0.
\]
For both $\lambda = \alpha_p$ or $\beta_p$, the sequence $(\lambda^m)_m$ belongs to the $2$-dimensional vector space $V$ of all sequences $(b_m)_m$ satisfying 
\[
b_{m+1} = a_p b_m - p b_{m-1}.
\]
Moreover, for $3 \nmid p$, the coefficients of $f_\cE(q)$ satisfy the recursion $a_{p^{m+1}} = a_p a_{p^m} - p a_{p^{m-1}}$ for all $m \geq 1$ due to the relations among the various Hecke operators, see for example \cite[\S VII.5.4]{serre-coursearithmetic}.
Hence $(a_{p^m})_m$ and $(a_{p^{m-1}})_m$ (the latter suitably interpreted as $0$ for $m=0$) also belong to $V$. It follows easily that for $q = p^m$ and all $m \geq 1$
\[
1 + q - \#{\cE(\bF_{q})} = (\alpha_p)^m + (\beta_p)^m = 2a_{p^m} - a_p a_{p^{m-1}} = 2a_q - a_p a_{q/p}.
\]
In conclusion the term $c^2_q$ in Theorem~\ref{thm:main} can be computed as 
\[
c^2_q = \frac{1}{q} \cdot (2a_q - a_p a_{q/p})^2
\]
as a combination of $q$-expansion coefficients of the newform $f_{\cE}(q)$. 
\end{rmk}

The geometric description of the parametrizing variety and thus also the computation of its cohomology can be adapted easily to the case of general scalar shifts $h \in \bF^\times_q$. The resulting fomula involves the cubic twist $\cE_h/\bF_q$ of the elliptic curve $\cE \otimes_{\bZ[1/3]} \bF_q$ over $\bF_q$ given by 
\[
hX^3 = Y ( Y-1).
\]
We set $c_{q,h} := \frac{1+ q - \#\cE_h(\bF_q)}{\sqrt{q}}$, so that $0 \leq c^2_{q,h} \leq 4$ by \eqref{eq:HB}. Note also that $c_{q,h}$ only depends on $h$ modulo cubes: the class of $h$ in $\bF_q^\times/(\bF_q^\times)^3$.

\begin{thm}
\label{thm:main_twisted}
Let $q$ be a prime power. For all $h \in \bF_q$, $h \not= 0$, we have
\begin{equation}
\label{eq:MTmod3version_twisted}
\pi(3,q;(0,h)) = \frac{1}{9} q^3 \cdot
\begin{cases}
1 - q^{-1} - 3q^{-2} & \mbox{if } q \equiv 0 \pmod 3,\\[1ex]
1 + \left(c^2_q  - 3\right) \cdot q^{-1} -2q^{-2} &  \mbox{if } q \equiv 1 \pmod 3 \mbox{ and $h$ is a cube in } \bF_q, \\[1ex]
1 + c_q \cdot c_{q,h} \cdot q^{-1} - 8q^{-2} &  \mbox{if } q \equiv 1 \pmod 3 \mbox{ and $h$ is not a cube in } \bF_q, \\[1ex]
1 - q^{-1} - 2q^{-2} & \mbox{if } q \equiv 2 \pmod 3.
\end{cases}
\end{equation}
\end{thm}

On avarage over all $h \in \bF_q^\times$ we obtain the predicted asymptotic up to $\rO(q^{-2})$ in the error term:
\[
\frac{1}{q-1} \sum_{h \in \bF_q^\times} \pi(3,q;(0,h)) =
\frac{1}{9} q^3 \cdot
\begin{cases}
1 - q^{-1} - 3q^{-2}, & \mbox{if } q \equiv 0 \pmod 3,\\[1ex]
1 - q^{-1} - 6q^{-2}, &  \mbox{if } q \equiv 1 \pmod 3,\\[1ex]
1 - q^{-1} - 2q^{-2}, & \mbox{if } q \equiv 2 \pmod 3.
\end{cases}
\]

\smallskip

\textbf{Method.} The proof of Theorem~\ref{thm:main} will be given at the end of 
Section~\S\ref{sec:proof of main thm} when enough of the geometry and cohomology of the problem has been explained. It then follows by combining the twisted Lefschetz trace formula of Proposition~\ref{prop:LTF}, which reduces the counting to the problem of determining the effect of Frobenius and the twisting $3$-cycle on cohomology,  with the explicit computation of the respective traces that is completed in Proposition~\ref{prop:traceFU}.

We will study the $\bF_q$-variety parametrizing pairs of algebraic elements over $\bF_q$ of general degree $d$ such that the respective minimal polynomials differ by $1$ . We count these pairs by means of the Lefschetz trace formula in \'etale cohomology. This strategy is applicable in principle for all $d$, but we succeed to pursue it only for $d\leq 3$ because here we manage to compute the respective cohomology as a representation. To summarize, the geometry for $d=3$ is as follows.

\begin{thm}
Let $q$ be a prime power and $3 \nmid q$. 
Roots of Cubic twin prime polynomial pairs in $\bF_q[T]$ are parametrized by a variety that is an $\bF_q$-form of a variety which is an $\bA^1$-bundle over the complement of an explicit divisor on an explicit hyperelliptic surface.
\end{thm}
\begin{proof}
This is the geometric content of Section~\S\ref{sec:geometry}. The hyperelliptic surface is covered by $E \times E$ for the elliptic curve $E: X^3 = Y(Y-1)$, and the divisor to  be removed is $T + H$ in the notation of loc.~cit.
\end{proof}

\textbf{Outline.} In Section~\S\ref{sec:parametervariety} we recall how to count rational $\bF_q$-points of a Galois-twisted variety and describe the variety $U_\tau$ over $\bF_q$ that enumerates twin prime polynomial pairs of any degree $d$ over $\bF_q$. In 
Section~\S\ref{sec:geometry} we analyse its geometry in the case $d=3$ with special emphasis on the action of a $3$-cycle $\tau$. 
Section~\S\ref{sec:cohomology} is devoted to the computation of (total) \'etale cohomology with compact support as a virtual representation in a Grothendieck-group of Galois representations together with an action by $\tau$. In the penultimate Section~\S\ref{sec:proof of main thm} we then evaluate the trace of the twisted Frobenius and derive the formula of 
Theorem~\ref{thm:main}.

\section{The singular series for function fields}
\label{sec:sing series}
Since we are interested in studying the number of prime polynomials of fixed degree $d$ as a function of the number of elements in the field $q$, we denote by $\pi_d(q) = \pi(d,q;(0))$ the number of primes of $\bF_q[T]$ of degree $d$. 
We would like to express the Hardy-Littlewood singular series  for the function field $\bF_q(T)$ 
\[
\fS_q((0,1)) = \prod_{\fp \mbox{ \tiny prime of } \bF_q[T]} \left(\frac{1-2N(\fp)^{-1}}{(1-N(\fp)^{-1})^2}\right) = \prod_{d \geq 1}  \left(\frac{1-2q^{-d}}{(1-q^{-d})^2}\right)^{\pi_d(q)}
\]
as the value of a convergent power series $S(u) \in \bQ[[u]]$ in $u = q^{-1}$. Here $S(u)$ shall be independent of $q$ and thus is uniquely determined if it exists. This can be done as follows. 
First, for $q = 2$ we have a local obstruction at places of degree $1$, hence $\fS_2((0,1)) = 0$.  Therefore we assume from now on that $q \geq 3$. 

\begin{lem}
\label{lem:gauss}
For all  $d \geq 1$ there is a polynomial $P_d(u) \in \bQ[u]$ such that for all prime powers $q$
\[
P_d(q^{-1}) = \pi_d(q) \cdot q^{-d}.
\]
Moreover, we have $\abs{P_d(t)} \leq 1$ for all complex $\abs{t} \leq 1$.
\end{lem}
\begin{proof}
This follows from Gau\ss's formula $\pi_d(q) = \frac{1}{d} \sum_{k \mid d} \mu(k) q^{d/k}$. Concretely, we have
\[
P_d(u) = \frac{1}{d} \sum_{k \mid d} \mu(k) u^{d-d/k}.
\]
The estimate is trivial because there are at most $d$ summands.
\end{proof}

\begin{lem}
\label{lem:log}
The following power series is convergent for $\abs{u} < 1/2$:
\[
\lambda(u) := \sum_{k \geq 2} \frac{2-2^k}{k} u^{k-1} \in u \cdot \bQ[[u]].
\]
For all $0<t_0 < 1/2$ there is a constant $c = c(t_0)$ such that $\abs{\lambda(t)} \leq c\cdot \abs{t}$ for all complex $\abs{t} \leq t_0$.
\end{lem}
\begin{proof}
The power series is an expansion of $\frac{\log(1-2u) - 2 \log (1-u)}{u}$ which is holomorphic in $\abs{u} < 1/2$, hence its radius of convergence is $1/2$.

The existence of a constant $c(t_0)$ and the estimate follows because $\lambda(u)/u$ is continuous and thus bounded on the compact ball of radius $t_0$.
\end{proof}

Since $\lambda(u^d)$ is divisible by $u^d$, the following sum formally converges 
\[
S(u) : = \exp\left(\sum_{d \geq 1} P_d(u) \cdot \lambda(u^d) \right) \in \bQ[[u]].
\]

\begin{prop}
The power series $S(u)$ converges for $\abs{u} < 1/2$. For all prime powers $q \geq 3$ we have
\[
S(q^{-1}) = \fS_q((0,1)).
\]
In particular, the product defining the function field singular series converges absolutely.
\end{prop}
\begin{proof}
It suffices to show that the sum of holomorphic functions $\sum_{d \geq 1} P_d(u) \cdot \lambda(u^d)$  converges uniformly on any ball of radius $t_0 < 1/2$. This follows at once from the estimates in Lemma~\ref{lem:gauss} and Lemma~\ref{lem:log}.

In order to compute the value $S(q^{-1})$ we rather compute its logarithm as follows:
\begin{align*}
\log S(q^{-1}) & = \sum_{d \geq 1} P_d(q^{-1}) \cdot \lambda(q^{-d}) = \sum_{d \geq 1} \pi_d(q)q^{-d} \cdot \frac{\log(1-2q^{-d}) - 2 \log (1-q^{-d})}{q^{-d}} \\
& = \sum_{d \geq 1} \pi_d(q) \cdot \log \left(\frac{1-2q^{-d}}{(1-q^{-d})^2} \right) = \log \prod_{d \geq 1}  \left(\frac{1-2q^{-d}}{(1-q^{-d})^2}\right)^{\pi_d(q)} = \log \fS_q((0,1)). \qedhere
\end{align*}
\end{proof}

\begin{rmk}
In order to analyse $S(u)$ at $u=1/2$, we use $P_1(u) = 1$ and split the factor corresponding to $d=1$ by writing
\[
S(u) = (1-2u)^2 \cdot \exp\Big(\sum_{k \geq 1} \big(2^{k+1}(\frac{1}{k} - \frac{1}{k+1}) + \frac{2}{k+1}\big)u^k \Big) \cdot \exp \left(\sum_{d \geq 2} P_d(u) \cdot \lambda(u^d) \right).
\]
The factor $\exp \left(\sum_{d \geq 2} P_d(u) \cdot \lambda(u^d) \right)$ converges for $u<1/\sqrt{2}$, and the middle factor can be dealt with at $u = 1/2$ by Abel's theorem of converging power series on their radius of convergence. Hence the series $S(u)$ converges at $u=1/2$ and takes the value $S(1/2) = 0$ there. This agrees with $\fS_2((0,1)) = 0$.
\end{rmk}

\begin{rmk}
It is not difficult to calculate the low degree terms of the power series $S(u)$. For all $m \geq 1$, we have in $\bQ[[u]]$, by truncating all of the power series in the definition, that 
\[
S(u) =  \sum_{\nu = 0}^m \frac{1}{\nu!} \left(\sum_{d=1}^m P_d(u) \sum_{k=2}^{\lfloor 1 + \frac{m}{d}\rfloor} \frac{2-2^k}{k} u^{d(k - 1)} \right)^\nu   + \rO(u^{m+1}).
\]
Concretely, we obtain using SageMath:
\[
S(u) = 1 - u - 2u^{2} - u^{3} - 2u^{4} + 2u^{5} + 6u^{7} + 7u^{8} + 13u^{9} + 20u^{10} + 32 u^{11} + 41u^{12}+ O(u^{13}).
\]
\end{rmk}

\begin{rmk}
\label{rmk:heuristicerrorterm}
When analysing the asymptotic $q \to \infty$ for fixed $d$, then of course the primes of degree $> d$ do not pose any constraints in the heuristic. Consequently, the correlation factor $\fS_q((0,1))$ should skip these primes. If we use the truncated power series 
\[
S(u)  = S_d(u) + \rO(u^{d+1}),
\]
with a polynomial $S_d(u)$ of degree $\leq d$, then it is easy to see that primes of degree $> d$ do not influence $S_d(u)$. Moreover, in order to take into acount the translation invariance for $h = (0,1)$, we should rather truncate modulo $u^d$. Hence we propose to compare the actual count with the truncated series $\frac{q^d}{d^2} \cdot S_{d-1}(q^{-1})$:
\[
\pi(d,q;(0,1)) = \frac{q^d}{d^2} \cdot S_{d-1}(q^{-1}) \cdot \big(1 + E(d,q;(0,1)) \big).
\]
Here are the predictions and error terms 
accordingly:
\[
{\renewcommand{\arraystretch}{1.2}
\setlength{\arraycolsep}{1em} 
\begin{array}{llllll}
\toprule
 d & \text{prediction} & q & \pi(d,q;(0,1))  & E(d,q;(0,1))  \\ \midrule
1 & q &                                 & q   &  0  \\ \midrule
2 &  \frac{q^2}{4}  (1- q^{-1}) &  \equiv  0  \ (2) & \frac{q^2}{4} (1 - 2q^{-1}) &  q^{-1} + \rO(q^{-2})\\ 
  & &  \equiv  1 \ (4) & \frac{q^2}{4} (1 - q^{-1})  &  0 \\
  & &  \equiv  3 \ (4) & \frac{q^2}{4} (1 - 3q^{-1})  &  2q^{-1} + \rO(q^{-2}) \\ \midrule
3 & \frac{q^3}{9} (1- q^{-1} - 2q^{-2}) &  \equiv  0 \ (3) & \frac{q^3}{9} (1 - q^{-1} - 3q^{-2})  &  q^{-2} + O(q^{-3})  \\
  & &   \equiv  1 \ (3) & \frac{q^3}{9} (1 - (c^2_q-3)q^{-1} - 2q^{-2})  \hspace{-0.5cm} &    	
(2-c_q^2)q^{-1} + \rO(q^{-2})
\\
  & &  \equiv  2 \ (3) & \frac{q^3}{9} (1 - q^{-1} - 2q^{-2})  &  0 \\ \bottomrule
\end{array}
}\]
Thus, if $d=3$ and  $q \not\equiv 1 \pmod 3$, then \eqref{eq:MTmod3version} is consistent with square root cancelation in \eqref{HLlargeqlimit}. 
When $q \equiv 1 \pmod 3$, the coefficient of $q^{-1}$ varies with $q$. We will now argue that at least on avarage among those $q$ this coefficient is again $-1$. 
There are unique angles $\teta_q \in [0,\pi]$ such that $c_q = 2 \cos(\teta_q)$. Since $\cE$ is a CM-curve, its $L$-function can be expressed in terms of a Hecke character following  Deuring \cite{deuring}. In this case already Hecke \cite{Hecke} showed 
that for primes split in the field of multiplication, here $\bQ(\sqrt{-3}) = \bQ(\zeta_3)$, the distribution of angles is uniform with respect to the measure $\frac{1}{\pi} d\teta$, see Section~\S2.4 of the survey by Sutherland \cite{sutherland} for more details and references. Hence, the value of $c_q^2$  on average is 
\[
\frac{1}{\pi} \int_0^\pi 4 \cos^2 (\teta) d\teta = 2,
\]
and  the average value of the coefficient of $q^{-1}$ 
becomes again $-1$.
\end{rmk}

\section{Parametrizing irreducible polynomials}
\label{sec:parametervariety}

In this section we describe the geometry of the parametrizing varieties of twin primes. 

\subsection{Rational points of twists} 
\label{sec:ratptstwist}

Let $\bar\bF_q$ be an algebraic closure of $\bF_q$. Let $X_0/\bF_q$ be a variety and denote by $X = X_0 \times_{\bF_q} \bar\bF_q$ or by the usual notation $X_{0,\bar\bF_q}$ the base change to $\bar\bF_q$. The $\bar\bF_q$-linear geometric $q$-Frobenius map  
\[
F : X \to X
\]
raises coordinates (defined over $\bF_q$, i.e., coordinates of $X_0$) to $q$-th powers. Hence the set of $\bF_q$-rational points of $X_0$ is the set of Frobenius fixed points of $X$
\begin{equation}\label{eq:fixedpts}
X_0(\bF_q) = \{ x \in X(\bar\bF_q) \ ; \ F(x) = x\},
\end{equation}
and 
counted by the Lefschetz trace formula in \'etale cohomology with compact support for 
$\ell \not= p$ as 
\begin{equation}
\label{eq:lefschetztraceformula}
\#{X_0(\bF_q)}  = \tr \big(F | \rH_{\rc}^*(X,\bQ_\ell)\big) := \sum_{i=0}^{2\dim X} (-1)^{i} \tr \big(F | \rH_{\rc}^i(X,\bQ_\ell)\big).
\end{equation}
For background on \'etale cohomology and a proof of the Lefschetz trace formula we refer to \cite{freitagkiehl}.

Forms of $X_0$ over $\bF_q$ are obtained by Galois descent via a twisted Galois action on $X$ by means of a continuous $1$-cocycle $\sigma \mapsto a_\sigma$
\[
a : \Gal(\bar\bF_q/\bF_q) \to \Aut(X/\bar\bF_q) 
\]
with values in the $\Gal(\bar \bF_q/\bF_q)$-module $\Aut(X/\bar \bF_q)$, see for example \cite[\S2]{sko:torsors}.
Being a cocycle means  for all $\sigma, \pi \in \Gal(\bar \bF_q/\bF_q)$ that  
$a_{\sigma \pi} = a_\sigma \circ \sigma(a_\pi)$. Equivalently, the map 
\[
\sigma \mapsto a_\sigma \sigma \in \Aut(X/\bF_q)
\]
defines another Galois action on $X$. The $a$-twist of $X_0/\bF_q$ is defined as the quotient of $X$ by the twisted Galois action and is here denoted by 
\[
X_{a}/\bF_q.
\]
Let $\ph \in \Gal(\bar \bF_q/\bF_q)$ be the arithmetic $q$-Frobenius map $\ph(z) = z^q$. The $1$-cocycle $a$ is uniquely determined by its value $\alpha = a (\ph)$. The Frobenius of $X$ arising from the new $\bF_q$-structure $X_{a}$ is nothing but 
\[
F_a = \alpha F.
\]
It follows that $\bF_q$-rational points of the twist are described by 
\begin{equation}\label{eq:fixedpts-twisted}
X_{a}(\bF_q)  = \{ x \in X(\bar \bF_q) \ ; \  F_a(x) = x\} = \{ x \in X(\bar \bF_q) \ ; \ F(x) = \alpha^{-1}(x)\}
\end{equation}
and counted as 
\begin{equation}
\label{eq:twistedLefschetztraceformula}
\#{X_{a}(\bF_q)}  = \tr \big(\alpha F | \rH_{\rc}^*(X,\bQ_\ell)\big).
\end{equation}

\subsection{Twisting varieties of polynomials} 
\label{sec:twistingvarpol}

Monic polynomials of degree $d$ 
\[
f(T) = T^d + a_1T^{d-1}+ \ldots + a_{d-1}T + a_d 
\] 
with coefficients $a_i \in \bF_q$ are parametrized by $d$-dimensional affine space as
\[
f = (a_1, \ldots, a_d) \in \bA^d(\bF_q).
\]
We have a finite 
branched $S_d$-cover $s: \bA^d \to \bA^d$ 
\[
s(x_1, \ldots,x_d) = (-\sigma_1(\underline{x}), \ldots, (-1)^d \sigma_d(\underline{x})) = \prod_{i=1}^d (T - x_i)
\]
whose coordinates  are 
given by the elementary symmetric polynomials in 
$\underline{x} = (x_1,\ldots,x_d)$
\[
\sigma_r(\underline{x}) = \sum_{1 \leq i_1 < \ldots < i_r \leq n} x_{i_1} \cdot \ldots \cdot x_{i_r}.
\]
The ramification locus of $s$ is given by the vanishing locus of the discriminant
\[
\Delta_{x} := \prod_{i \not= j} (x_i-x_j) \in \bZ[a_1, \ldots, a_d] \subseteq \bZ[x_1, \ldots, x_d],
\]
that is the locus of polynomials with multiple roots.

\medskip

We consider the twist\footnote{Because of the general theorem Hilbert 90, $\rH^1(\bF_q,\GL_d) = 1$, the twisted $\bA^d$ is isomorphic to $\bA^d$. This fact leads to good asymptotic formulae for the number of irreducible polynomials in $\bF_q[X]$ of degree $d$.} of $\bA^d$ by the $1$-cocycle $\tau$ determined by its value on Frobenius being the $d$-cycle
\[
\tau(\ph) = (123 \ldots d) \in S_d
\]
with respect to the permutation representation
\[
S_d \to \GL_d(\bF_q) \to \Aut(\bA^d).
\]
As described in \eqref{eq:fixedpts-twisted}, rational points of the twist are
those points $\underline{x} = (x_1, \ldots, x_d) \in \bA^d(\bar \bF_q)$ invariant under the twisted Frobenius (with indices considered modulo $d$): 
\[
(\bA^d)_\tau(\bF_q) = \{\underline{x} \in \bA^d(\bar \bF_q) \ ; \ x_i^q  = x_{i+1} \  \text{ for all } i=1,\ldots, d\}.
\]
Being $S_d$-invariant, the cover $s$ becomes a cover
\[
s_\tau :  (\bA^d)_\tau \to \bA^d.
\]
Rational points of the twist outside of the discriminant locus 
map under $s$ to polynomials whose roots are all distinct and cyclically permuted by $q$-Frobenius, hence exactly to irreducible polynomials. Every unramified point $f \in \bA^d(\bF_q)$ in the image of the map 
\[
s_\tau : (\bA^d)_\tau(\bF_q) \to \bA^d(\bF_q)
\]
has $d$ preimages, because each preimage in $(\bA^d)_\tau(\bF_q)$ is determined by its $x_1 \in \bar \bF_q$ which can be any of the $d$ distinct roots of the irreducible polynomial $f$.

\subsection{Parametrizing twin prime polynomials} 
\label{sec:parametrizetpp}

Let $M$ be the $\bF_q$-vector space with generators 
\[
x_1, \ldots, x_d, y_1, \ldots, y_d, z
\]
and the only relation $\sigma_1(\underline x) = \sigma_1(\underline{y})$. 
The corresponding projective space
\[
\bP(M) = \Proj (\Sym^\bullet M) 
\]
is the hyperplane $V\big(\sigma_1(\underline x) = \sigma_1(\underline{y})\big)$ in $\bP^{2d}$, the projective space with homogeneous coordinates 
\[
[x_1: \ldots : x_d:y_1: \ldots : y_d: z].
\]
In $\bP(M)$ we consider the subvariety 
\[
X = V\big(\sigma_2(\underline{x}) -\sigma_2(\underline{y}), \ \ldots , \sigma_{d-1}(\underline{x}) - \sigma_{d-1}(\underline{y}), \ \sigma_d(\underline{x}) - \sigma_d(\underline{y}) + (-z)^d\big) \subseteq \bP(M).
\]
On the distinguished open $z \not= 0$ we map a point $[\underline x : \underline y: 1]$ to the pair of polynomials
\[
f(T) = \prod_{i=1}^d (T - x_i), \qquad g(T) = \prod_{j=1}^d (T-y_j)
\]
by means of which the defining equations for $X$ take the simple form
\[
g(T) = f(T)  + 1.
\]
Let $\Delta_x$ (resp.\ $\Delta_y$) denote the discriminant in terms of the tuple of variables $\underline x$ (resp.\ $\underline y$). Let $U \subseteq X$ be the open 
\[
U = X \cap \{ \Delta_x \cdot \Delta_y \cdot z \not= 0\},
\]
i.e., the locus where the description is by separable polynomials. 

We let $S_d$ act diagonally on $M$ by permuting blockwise both tuples of variables $\underline {x}$ and $\underline y$, keeping $z$ fixed. This induces actions on $\bP(M)$, $X$ and $U$ respectively. 
It follows from  Section~\ref{sec:twistingvarpol} and the notation introduced there that rational points of the twist $U_\tau$ are 
\[
U_\tau(\bF_q) = \{(x_1, \ldots, y_d) \in \bar \bF_q^{2d} \ ; \ f \text{ and } g \text{ are monic,  irreducible and } g(T) = f(T) + 1\}.
\]
Therefore we have proved the following proposition as a consequence of \eqref{eq:lefschetztraceformula}. Recall that we denote by $U_{\bar \bF_q}$ the base change $U \times_{\bF_q} \bar \bF_q$.

\begin{prop}
\label{prop:LTF}
The number 
of twin prime polynomial pairs in $\bF_q[T]$ of degree $d$ is
\[
\pi(d,q;(0,1)) 
= \frac{1}{d^2} \cdot \#{U_\tau(\bF_q)} = \frac{1}{d^2} \tr\big(\tau F | \rH^\ast_{\rc}(U_{\bar \bF_q},\bQ_\ell)\big) .
\]
\end{prop}

We are left with the task to understand $\rH^\ast_{\rc}(U_{\bar \bF_q},\bQ_\ell)$ as $S_d \times \Gal(\bar \bF_q/\bF_q)$-module. Since we are only interested in a certain trace, it suffices to determine cohomology in a Grothendieck group of virtual representations.

\begin{ex}
\label{ex:caseD=2}
We compute the case $d=2$ using the above notation along the strategy for the $d=3$ case to be treated in the following section \S\ref{sec:geometry}. The variety $X$ is given in $\bP^4$ by the equations
\begin{align*}
x_1 + x_2 &= y_1 + y_2, \\
x_1 x_2 &= y_1y_2 - z^2.
\end{align*}
Using $a = x_1-y_2$, $b = x_2 - y_2$, $c = y_1-y_2$,  and $z$ we may express these as
\begin{align*}
a+ b  & = c, \\
ab & =  -z^2,
\end{align*}
so that the complement of $\star_M := [1:1:1:1:0]$ in $X$ forms an $\bA^1$-bundle over  the smooth conic
\[
\bP^1 \simeq  Y = V(z^2 + ab) \subseteq \bP^2.
\]
The open $U = X \cap \{(x_1-x_2)(y_1-y_2)z \not= 0\}$ is the preimage in this $\bA^1$-bundle of 
\[
V = Y \cap \{(a-b)(a+b)z \not= 0\}.
\]
In coordinates $[a:b:z]$, this removes the following set of points $D$ from $Y$:
\begin{itemize}
\item
If the characteristic is $2$,  the points $D  =\{P_1,P_2,P_3\}$ with 
\[
P_1 = [0:1:0], \ P_2 = [1:0:0], \text{ and } P_3 = [1:1:1].
\]
\item
If the characteristic is not $2$,  the points $D = \{P_1,P_2,Q_+,Q_-,R_+,R_-\}$ with 
\begin{align*}
& P_1 = [0:1:0], \ P_2 = [1:0:0], \\
& Q_{+} = [1:1: i], \ Q_- = [-1:-1:i], \\
& R_+= [1:-1:1], \ R_- = [-1:1:1].
\end{align*}
Here $i$ denotes a square root of $-1$, possibly in a quadratic extension of $\bF_q$.
\end{itemize}
The involution $\tau$ defined by $x_1 \leftrightarrow x_2$, $y_1 \leftrightarrow y_2$ acts via $a \mapsto -a$ and $b \mapsto -b$, hence the points $P_k$, for $k=1, \ldots, 3$, are fixed while the points with index $\pm$ are swapped. The action of Frobenius on these points is only nontrivial for $Q_{\pm}$, swapping these points, if $-1$ is not a square in $\bF_q$, i.e., if the Jacobi-symbol $\qrB{-1}{q} = (-1)^{(q-1)/2}$ equals $-1$ (here $q$ is odd).

In summary, we may evaluate the trace formula of Proposition~\ref{prop:LTF} as
\begin{align} 
\label{eq:caseD=2}
\pi(2,q;(0,1)) & = \frac{1}{4} \tr\big(\tau F | \rH^\ast_{\rc}(U_{\bar \bF_q},\bQ_\ell)\big)  = \frac{1}{4}q \cdot \tr\big(\tau F | \rH^\ast_{\rc}(V_{\bar \bF_q},\bQ_\ell)\big) \notag \\
& = \frac{1}{4}q \cdot \Big(\tr\big(\tau F | \rH^\ast(\bP^1_{\bar \bF_q},\bQ_\ell)\big) - \tr\big(\tau F| \rH^\ast(D_{\bar \bF_q},\bQ_\ell)\big) \Big) \notag \\
& = \frac{1}{4}q \cdot \Big(1+q  -  \# D_\tau(\bF_q) \Big) \notag \\
& = \frac{1}{4}q^2 \cdot
  \begin{cases}
 1- (2 - \qrB{-1}{q}) q^{-1}  & q \text{ is odd}, \\
 1- 2q^{-1} & q \text{ is even}.
\end{cases}
\end{align}
\end{ex}

\section{Geometry}
\label{sec:geometry} 

We keep the notation of Section~\ref{sec:parametrizetpp} but specialize to $d=3$. Moreover, we consider variables 
\[
\underline{x}  = (x_0,x_1,x_\infty) \quad \text{ and } \quad \underline{y} = (y_0,y_1,y_\infty)
\]
in order to simplify notation when dealing with symmetries. The reader is advised to consult with Figure~\ref{diagram} which provides a visualizations of the construction.  

\subsection{Linear projection}

Recall that $M$ is the $\bF_q$-vector space spanned by $x_0,x_1,x_\infty,y_0,y_1,y_\infty,z$ subject to the relation $\sigma_1(\underline{x}) = \sigma_1(\underline{y})$. 
We introduce the following subspaces 
\[
M \supseteq N \supseteq L
\]
with $L$ generated by the images of $x_i - y_j$ for all $i,j \in \{0,1,\infty\}$ and $N$ generated by $L$ and the image of $z$. The space $L$ has a basis $a,b,c,d$ given by the matrix entries of 
\[
\matzz{a}{b}{c}{d} = \matzz{y_\infty-x_1}{x_\infty-y_1}{x_0 - y_1}{y_0 - x_1}.
\]

\begin{lem}
\label{lem:lks}
We will need the following descriptions with respect to the basis $a,b,c,d$:
\[
\begin{array}{lcl@{ \quad }|@{ \quad }lcl@{ \quad }|@{ \quad }lcl}
x_\infty - y_\infty  & = & d - c & x_\infty - y_0 & = & a - c & x_\infty - y_1 & = & b \\
x_0 - y_\infty  & = &  d-b &  x_0 - y_0 & = & a-b &  x_0 - y_1 & = &  c \\
x_1 - y_\infty  & = &  -a &  x_1 - y_0 & = & -d  &  x_1 - y_1 & = & b+c - a - d \\
\\
& &  &  y_\infty  - y_1 & = & b+c - d   & y_\infty  - y_0 & = & a - d \\ 
x_\infty - x_1 & = &a + d - c & &  & & y_0- y_1 & = & b+c - a \\
x_\infty - x_0 & = & b - c & x_0- x_1 & = & a + d - b & & &
\end{array}
\]
\end{lem}
\begin{proof}
This is elementary, for example, recall that $\sigma_1(\underline{x}) = \sigma_1(\underline{y})$, hence
\[
x_\infty - y_\infty  = y_0 + y_1 - x_0 - x_1 = d - c.
\]
Moreover, the symmetry $x_i \longleftrightarrow y_i$ for $i=0,1,\infty$ translates to the involution
\[
a \longleftrightarrow b \quad \text{ and } \quad c \longleftrightarrow d.
\]
This helps deducing other linear combinations from known ones.
\end{proof}

In each step of $M \supseteq N \supseteq L$ the dimension drops by $1$ and the induced rational maps 
\[
\bP(M) \dasharrow \bP(N) \dasharrow \bP(L) 
\]
are each defined outside one point. More precisely, let 
\[
\star_M = [1:1: \ldots : 1: 0]
\]
be the point in $\bP(M) \subseteq \bP^6$ where all linear coordinates in $N$ vanish, then linear projection is a geometric line bundle (Zariski-locally isomorphic to $\bP(N) \times \bA^1$)
\[
\bP(M) \setminus \star_M \to \bP(N).
\]

The equations for $X \subseteq \bP(M)$ allow the following manipulations:
under the assumption $\sigma_i(\underline x) = \sigma_i(\underline y)$ for $i = 1,2$ a quick calculation shows that we have  
\begin{align}
\sigma_3(\underline x) = \sigma_3(\underline y) + z^3 
\iff & (T-y_0)(T-y_1)(T-y_\infty) = z^3 + (T-x_0)(T-x_1)(T-x_\infty)  \notag \\
\iff & (x_\infty-y_0)(x_\infty-y_1)(x_\infty-y_\infty)  = z^3 \notag \\
\iff & (a-c)b(d-c) = z^3. \label{eq:dritte}
\end{align}
Moreover we find 
\begin{align}
\det \matzz{a}{b}{c}{d} & = (y_\infty-x_1)(y_0-x_1) - (x_\infty- y_1)(x_0-y_1) \notag \\
& = y_\infty y_0 - x_1(y_\infty+y_0) + x_1^2 - y_1^2 - x_\infty x_0 + y_1(x_\infty +x_0) \notag \\
& = \sigma_2(\underline{y}) - \sigma_2(\underline{x})  + (x_1+y_1)(\sigma_1(\underline{x}) - \sigma_1(\underline{y})) \notag \\
& = \sigma_2(\underline{y}) - \sigma_2(\underline{x}).  \label{eq:zweite}
\end{align}
Hence, we define $Y \subseteq \bP(N)$ by the equations
\[
Y := V((a-c)b(d-c) = z^3, \ ad-bc = 0) \subseteq \bP(N).
\]

\begin{prop}
\label{prop:linebundle}
The restriction of the linear projection $\bP(M) \dasharrow \bP(N)$ to $Y$ is a geometric line bundle 
\[
\pr_N: X \setminus \star_M \to Y.
\]
\end{prop}
\begin{proof}
Immediately from the computations  \eqref{eq:dritte} and \eqref{eq:zweite}.
\end{proof}

\subsection{A torsor}

We now analyse the second linear projection $\pr_L: \bP(N) \dasharrow \bP(L)$, which in coordinates $[a:b:c:d:z]$ is defined outside of the point 
\[
\star_N  = [0: \ldots : 0: 1].
\]
Because $\star_N \notin Y$ (the cubic equation fails) the second linear projection restricts to a finite map 
\[
j = \pr_L|_Y : Y \to Z  := V(ad-bc = 0) \subseteq \bP(L).
\]
The quadric $Z$ is isomorphic to $\bP^1 \times \bP^1$ by the following isomorphism
\begin{align*}
\bP^1 \times \bP^1 & \xrightarrow{\sim} Z \\
([u:v], [r:s]) & \mapsto \matzz{a}{b}{c}{d} = \binom{u}{v} \cdot (r,s).
\end{align*}
For later use we also record the inverse map as
\begin{equation}
\label{eq:bihomogeneous_coordinates_quadric}
[u:v] = [a:c] = [b:d] , \qquad [r:s] = [a:b] = [c:d].
\end{equation}
The map $j: Y \to Z$ has the structure of a ramified $\mu_3$-torsor on Z defined by the equation 
\[
z^3 = (a-c) b(d-c).
\]
The $\mu_3$-action is given by  $z \mapsto \zeta \cdot z$ for $\zeta \in \mu_3$. In the bihomogeneous coordinates $[u:v],[r:s]$ the equation for the torsor becomes
\[
z^3 = (ur - vr)us (vs - vr) = uv(u-v) \cdot rs (s-r).
\]
This allows to read off the branch locus  as
\[
S = \{0,1,\infty\} \times \bP^1 \cup \bP^1 \times \{0,1,\infty\} \subseteq \bP^1 \times \bP^1,
\]
and that $j$ induces an isomorphsm
\begin{equation}
\label{eq:compareST}
j|_T : \ T :=j^{-1}(S)_{\redu} \xrightarrow{\sim} S.
\end{equation}
of $S$ with the reduced preimage in $Y$ of the branch locus.

\subsection{Symmetries and automorphisms of the projective line}
\label{sec:symm}

We consider the symmetric group $S_3$ as the group of permutations of the set $\{0,1,\infty\}$. It acts
naturally on variables $\underline{x}$ (homogeneous linear coordinate functions) by the right action
\begin{equation}
\label{eq:permutatevariables}
\sigma^\ast(x_i) = x_{\sigma^{-1}(i)}
\end{equation}
and similarly for $\underline{y}$. Since we let $z$ be fixed by $S_3$, we obtain induced $S_3$-actions on $L \subseteq N \subseteq M$ and furthermore on 
\[
X \setminus \star_M \to Y \to Z.
\]
There is also a natural $S_3$-action on $\bP^1$ by projective linear transformations permuting the subset $\{0,1,\infty\} \subseteq \bP^1(\bF_q)$ accordingly. For example, the $3$-cycle 
\[
\tau = (01\infty)
\]
acts on a parameter $\lambda$ for $\bP^1$ as 
\begin{equation}
\label{eq:tauonP1}
\tau^\ast(\lambda) = \frac{1}{1-\lambda}.
\end{equation}
Since counting cubic twin prime polynomials requires control of the effect on cohomology of the $3$-cycle $\tau$, we compute its effect on geometry more explicitly in coordinates: 

\begin{lem}
\label{lem:actionoftauonZ}
The $3$-cycle $\tau = (01\infty)$ acts on $X$, $Y$ and $Z = \bP^1 \times \bP^1$ respectively as
\begin{align}
\tau([x_0:x_1:x_\infty:y_0:y_1:y_\infty:z]) & = [x_\infty:x_0:x_1:y_\infty:y_0:y_1:z], \label{eq:tauonX} \\
\tau([a:b:c:d:z]) & = [-c:-d:a-c:b-d:z], \label{eq:tauonY} \\
\tau([u:v],[r:s]) & = ([-v:u-v],[r:s]). \label{eq:tauonZ}
\end{align}
In particular $\tau$ acts on $Z$ by $\tau \times \id$ for the natural action of $\tau$ on $\bP^1$.
\end{lem}
\begin{proof}
This follows from \eqref{eq:permutatevariables} and Lemma~\ref{lem:lks}:
\begin{align*}
\tau^\ast(a) & = \tau^\ast(y_\infty - x_1) = y_1 - x_0 = -c, \\
\tau^\ast(b) & = \tau^\ast(x_\infty - y_1) = x_1 - y_0 = -d, \\
\tau^\ast(c) & = \tau^\ast(x_0-y_1) = x_\infty - y_0 = a-c, \\
\tau^\ast(d) & = \tau^\ast(y_0-x_1) = y_\infty - x_0 = b-d. 
\end{align*}
 For $Z \simeq \bP^1 \times \bP^1$ we use rational parameters, see \eqref{eq:bihomogeneous_coordinates_quadric}, 
\begin{equation} \label{eq:lambdamu}
\lambda = \frac{u}{v} = \frac{a}{c} = \frac{b}{d}, \qquad \mu = \frac{r}{s} = \frac{a}{b} = \frac{c}{d}
\end{equation}
for the two factors $\bP^1$. 
\begin{align*}
\lambda & \mapsto \ \frac{\tau^\ast(a)}{\tau^\ast(c)} 
= \frac{-c}{a-c} = \frac{-v}{u-v} = \frac{1}{1-\lambda}  = \tau^\ast(\lambda), \\
\mu &  \mapsto  \ \frac{\tau^\ast(a)}{\tau^\ast(b)} 
= \frac{-c}{-d} = \mu.  \qedhere
\end{align*}
\end{proof}

\subsection{Enters the elliptic curve}

The cubic curve $E$ given by the cubic equation 
\[
E  = \{w^3 = uv (u-v)\} 
\]
is a branched $\mu_3$-torsor (with $\zeta \in \mu_3$ acting by $w \mapsto \zeta w$)
\[
\pi:  E \to \bP^1,\qquad \pi([u:v:w]) = [u:v].
\]
 The curve $E$ is an elliptic curve unless $p=3$ when it is rational. In any case, there are unique $\bF_q$-rational points $P_0,P_1,P_\infty \in E(\bF_q)$ with $\pi(P_i) = i$ for all $i \in \{0,1,\infty\}$. These points all lie on the line $w=0$:
\[
P_0 = [0:1:0], \quad P_1 = [1:1:0], \quad P_\infty = [1:0:0].
\] 

\begin{prop}
\label{prop:tauonE}
Let $p \not= 3$, and consider $E$ as an elliptic curve with $P_0$ as the $0$ for the group law. 
\begin{enumerate}
\item $P_1$ and $P_\infty$ are $3$-torsion elements with $P_\infty = 2 P_1$.
\item 
\label{propitem:tauonE}
Translation by $P_1$ defines an isomorphism $\tau_E: E \to E$  of order $3$
\[
\tau_E([u:v:w]) = [-v:u-v:w].
\]
lifting $\tau : \bP^1 \to \bP^1$, the action by the $3$-cycle $\tau = (01\infty)$ defined by \eqref{eq:tauonP1}. 
\item 
The $\mu_3$-torsor $E \to \bP^1$ ramifies exactly in $\{P_0,P_1,P_\infty\}$, i.e., we have a finite \'etale covering
\[
\pi : E \setminus \{P_0,P_1,P_\infty\} \to \bP^1 \setminus \{0,1,\infty\}.
\]
\end{enumerate}
\end{prop}
\begin{proof}
(1) The line $w = 0$ intersects $E$ in the divisor $P_0 + P_1 + P_\infty$, thus $P_\infty = - P_1$ with respect to the group law with $P_0 = 0$. The map $\pi : E \to \bP^1$ is totally ramified in $0,1,\infty$. Hence, for all $i,j \in \{0,1,\infty\}$, the divisor  $3(P_i - P_j)$ is a difference of fibres and thus linearly equivalent to $0$. It follows that with respect to the group structure on $E$ we have $3P_1 = 3P_\infty = 0$ and $P_\infty = 2 P_1$.

(2) The points $P_0,P_1, P_\infty$ are $\mu_3$-invariant. Therefore $\mu_3$ acts via automorphisms of $E$ as a group and translation by $P_1$ commutes with the action by $\mu_3$. Hence there is a unique commutative square
\[
\xymatrix@M+1ex{
E \ar[r]^{\tau_E} \ar[d]^\pi  & E \ar[d]^\pi  \\
\bP^1 \ar@{.>}[r]^{\tilde{\tau}} & \bP^1 }
\]
Since $\tau_E$ permutes $P_0 \leadsto P_1 \leadsto P_\infty \leadsto P_0$, the induced map 
$\tilde{\tau}$ does likewise with $0 \leadsto 1 \leadsto \infty \leadsto 0$ and thus agrees with $\tau$ acting by \eqref{eq:tauonP1} on $\bP^1$.

We now determine the formula for $\tau_E$.  
The Weierstra\ss-equation of the elliptic curve $E$ has the form 
\[
y^2 - y = x^3
\]
with $x = -w/u$ and $y = v/u$. In these coordinates $P_0 = [0:1:0]$ becomes the point at infinity.
Addition with $P_1 = (0,1)$ in $xy$-coordinates takes the form, see \cite{silverman} group law algorithm 2.3 for formulas, 
\[
\tau_E([u:v:w]) = \tau_E(x,y)  = (-\frac{x}{y},1-\frac{1}{y}) = (\frac{-w}{-v}, \frac{u-v}{-v}) = [-v:u-v:w]. 
\]
(3) This is clear: by the Riemann-Hurwitz formula there is no ramification left. 
\end{proof}

We now consider a second copy of $E$ with coordinates $[r:s:t]$ and equation 
\[
t^3 = rs(r-s)
\]
The $\mu_3$-torsor $j : Y \to Z$ is dominated by the (branched) $\mu_3 \times \mu_3$-torsor 
\[
\pi \times \pi: \ Y' := E \times E \to \bP^1 \times \bP^1 \simeq Z
\]
by means of the map $j': Y'  \to Y$ in $[a:b:c:d:z]$-coordinates
\begin{equation}
\label{eq:maphprime}
j'([u:v:w], [r:s:t]) =  [ur:us:vr:vs:-wt].
\end{equation}
This identifies $Y$ with the quotient by the 
antidiagonal $\mu_3$-action on $Y'=E \times E$, i.e., with respect to $\zeta \in \mu_3$ acting by 
\[
([u:v:w], [r:s:t])  \mapsto ([u:v: \zeta w], [r:s:\zeta^{-1}t]).
\]

We let $\tau$ act on $E \times E$ by 
\[
\tau := \tau_E \times \id = \Big(([u:v:w], [r:s:t]) \mapsto ( [-v:u-v:w] , [r:s:t])  \Big).
\]

\begin{cor}
Let $p \not= 3$.
The map $j' : Y' \to Y$ is $\tau$-equivariant.
\end{cor}
\begin{proof}
This follows from \eqref{eq:tauonY}, \eqref{eq:maphprime} and Proposition~\ref{prop:tauonE}\eqref{propitem:tauonE}.
\end{proof}

The reduced preimage in $Y' = E \times E$ of the branch locus   is
\[
T' = j'^{-1}(T)_{\redu} = (\pi \times \pi)^{-1}(S)_{\redu} = \{P_0,P_1,P_\infty\} \times E \cup  E \times \{P_0,P_1,P_\infty\} \subseteq E \times E.
\]

\subsection{The open part}

We are ultimately interested in twists of 
\[
U = X \setminus \{\Delta_x \cdot \Delta_y \cdot z = 0\}.
\] 
Because of Lemma~\ref{lem:lks} we set 
\[
V := Y \setminus V((b-c)(a+d-c)(a+d-b)(a-d)(b+c-d)(b+c-a)z = 0),
\]
and immediately obtain the following corollary from Proposition~\ref{prop:linebundle}.

\begin{cor}
The linear projection $\bP(M) \dasharrow \bP(N)$ restricts to a geometric line bundle 
\[
\pr_N: U \to V.
\]
\end{cor}

The image of $V \subseteq Y$ in $Z \simeq \bP^1 \times \bP^1$ can best be described in terms of rational parameters $\lambda$ and $\mu$, see \eqref{eq:lambdamu}. The locus $z = 0$ is the preimage of $S$. The locus $\Delta_x = 0$ is given by the equation
\[
0 = \Delta_x = (b-c)(a+d-c)(a+d-b) = (us - vr)(ur+vs - vr)(ur + vs - us)
\]
which adds outside of $S$ the following divisors 
\[
\{\mu = \lambda\}, \qquad \{\mu = \frac{1}{1-\lambda}\}, \qquad \{\mu = 1 - \frac{1}{\lambda}\}.
\]
For $\Delta_y = 0$ we obtain in addition
\[
\{\mu = \frac{1}{\lambda}\}, \qquad \{\mu = 1-\lambda\}, \qquad \{\mu =  \frac{\lambda}{\lambda-1}\}.
\]
These are the graphs
\[
\Gamma_\sigma \subseteq \bP^1 \times \bP^1 \simeq Z
\]
of all the automorphisms $\sigma: \bP^1 \to \bP^1$ for all $\sigma \in S_3$ with respect to the action recalled in Section~\ref{sec:symm}. We set
\[
\Gamma = \bigcup_{\sigma \in S_3} \Gamma_\sigma.
\]
Now $V$ is the preimage under $j:Y \to Z$ of the open 
\[
W := Z \setminus (S \cup \Gamma).
\]
We furthermore set 
\[
H = \bigcup_{\sigma \in S_3} H_\sigma, \qquad \text{ with } H_\sigma := j^{-1}(\Gamma_\sigma)_{\redu},
\]
and similarly $H'$ and $H'_{\sigma}$ as reduced preimages of $H$ and $H_\sigma$ under $j': E \times E \to Y$.

\begin{lem}
\label{lem:tauonHGamma}
For all $\sigma \in S_3$ we have 
\[
\tau(H_{\sigma}) = H_{\sigma \tau^{-1}}, \quad \text{ and } \quad \tau(\Gamma_\sigma) = \Gamma_{\sigma \tau^{-1}}.
\]
\end{lem}
\begin{proof}
This is a consequence of abstract nonsense on graphs.
\end{proof}

\subsection{Intersections of components}
For the curves $C = S, \Gamma, T, H, T'$ and $H'$ we denote by $C_0$ the subvariety of points that lie in at least two components. For a point $P \in \Gamma_0 \setminus S$ we set
\[
n_P : = \#{\{\sigma \ ; \ P \in \Gamma_\sigma\}}.	
\]
Moreover, we set as reduced subvarieties
\[
\Gamma_n = \bigcup_{P \in \Gamma_0 \setminus S, \ n_P = n} P.
\]
Although the individual $P$ might not be defined over $\bF_q$, their union with a fixed number $n_P$ is.
\begin{lem}
\label{lem:geometryoffixedpoints}
In rational parameters $\lambda, \mu$ of $Z=\bP^1 \times \bP^1$ we have the following description of $\bar \bF_q$-rational points:
\begin{align*}
\Gamma_2(\bar \bF_q) & = \left\{
\begin{array}{cl}
\{-1, \frac{1}{2},2\} \times  \{-1, \frac{1}{2},2\} \qquad & p \not= 2,3, \\
\emptyset & p = 2, 3,
\end{array}
\right. \\[1ex]
\Gamma_3(\bar \bF_q) & = \left\{
\begin{array}{cl}
\{-\zeta_3,-\zeta_3^2\} \times \{-\zeta_3,-\zeta_3^2\} \quad & p \not= 3, \\
\emptyset & p=3,
\end{array}
\right. \\[1ex]
\Gamma_6(\bar \bF_q) & = \left\{
\begin{array}{cl}
\emptyset & p \not= 3, \\
\{(-1,-1)\}  \quad \qquad \qquad \qquad & p = 3. \\
\end{array}
\right. 
\end{align*}
\end{lem}
and $\Gamma_n(\bar \bF_q) = \emptyset$ for $n = 4$, $5$ and $n \geq 7$.
\begin{proof}
If $P = (\lambda_0,\mu_0)  \in \Gamma_\sigma \cap \Gamma_{\sigma'}$ with $\sigma \not= \sigma'$, then $\sigma'(\lambda_0) = \mu_0 = \sigma(\lambda_0)$. Hence $\lambda_0$ and therefore $\mu_0$ are fixed points for non-trivial elements of the natural $S_3$ action on $\bP^1$. The fixed points are easily listed and catalogued according to the size of the stabilizer as given in the lemma.
\end{proof}

In the same way we define for a point $P \in H_0 \setminus T$ 
\[
n_P : = \abs{\{\sigma \ ; \ P \in H_\sigma\}},
\]
and set as reduced subvarieties
\[
H_n = \bigcup_{P \in H_0 \setminus T, \ n_P = n} P.
\]
Since $H_\sigma = j^{-1}(\Gamma_\sigma)$ we find $n_P = n_{j(P)}$ for all $P \in H_0 \setminus T$, in particular 
\[
H_n = j^{-1}(\Gamma_n).
\]

\begin{lem}
\label{lem:H3}
Let $p \not= 3$. Then, in terms of coordinates $\xi = z/(vs), \lambda = u/v, \mu = r/s$,
\[
H_3(\bar \bF_q) \simeq \{-1,-\zeta_3,-\zeta_3^2\} \times  \{-\zeta_3,-\zeta_3^2\} \times \{-\zeta_3,-\zeta_3^2\}.
\]
The induced action of $\tau F$ on $H_3(\bar \bF_q)$ has no fixed points.
\end{lem}
\begin{proof}
The fibres above $(\lambda,\mu)$ with $\lambda,\mu \not= \infty$ are roots of 
\[
\xi^3 = \lambda (\lambda - 1) \cdot \mu (1-\mu).
\]
Evaluating for $(\lambda,\mu) \in \{-\zeta_3,-\zeta_3^2\} \times \{-\zeta_3,-\zeta_3^2\}$ results in each case in the equation
\[
\xi^3 = - 1,
\]
hence the description of $H_3$ as given in the lemma.

The points $\lambda = -\zeta_3, -\zeta_3^2$ are the fixed points for $\tau$ acting on $\bP^1$. The action on $\xi$ is as follows:
\[
\tau^\ast(\xi) = \frac{\tau^\ast(z)}{\tau^\ast(vs)} = \frac{z}{(u-v)s} = \xi \frac{v}{u-v} = \xi \frac{1}{\lambda - 1}.
\]
For $\lambda = -\zeta_3, -\zeta_3^2$ the factor $1/(\lambda -1)$ equals $-\lambda \in \mu_3$. If $(\xi,\lambda,\mu)$ is a fixed point under $\tau F$, then
\[
(\xi,\lambda,\mu) = \tau F(\xi, \lambda,\mu) = \tau(\xi^q,\lambda^q, \mu^q) = (-\lambda^q \xi^q, \lambda^q, \mu^q).
\]
Hence $\lambda^q = \lambda$, i.e., $\zeta_3 \in \bF_q^\times$, hence also $\xi^q = \xi$, and therefore equating the first coordinate yields $\lambda = -1$, a contradiction.
\end{proof}

\subsection{Summary}

Here is a diagram summarizing the varieties and maps considered above:
\begin{figure}[H]
\[
\xymatrix@M+1ex{
\bP(M) \ar@{-->}[dd] \ar@{}[r]|{\displaystyle \supseteq} & X \setminus \star_M \ar[dd]_{\pr_M} & 
\ar@{}[r]|{\displaystyle \supseteq}  & U \ar[dd] & &  \\
& & Y'=E \times E \ar[dd]_{\pi \times \pi} \ar[dl]_{j'}  & \ar@{}[r]|{\displaystyle \supseteq}  & V'  \ar[dl]  &
T' \cup H' \\
\bP(N)  \ar@{-->}[d] \ar@{}[r]|{\displaystyle \supseteq} & Y \ar[d]_{j} & \ar@{}[r]|{\displaystyle \supseteq}  & V \ar[d] & T \cup H & \\ 
\bP(L) \ar@{}[r]|{\displaystyle \supseteq} & Z \ar@{}[r]|{\displaystyle \simeq} & \bP^1 \times \bP^1 
\ar@{}[r]|{\displaystyle \supseteq}  & W & S \cup \Gamma & 
}
\]
\label{diagram}
\caption{Diagram of varieties used to describe $U$}
\end{figure}
Note that the $3$-cycle $\tau = (01\infty)$ acts in a compatible way on the diagram, including compatibility with the torsor structures. 

\section{Cohomology}
\label{sec:cohomology}

\subsection{Cohomology with values in a Grothendieck group}

The $3$-cycle $\tau = (01\infty)$ generates the alternating group $A_3 \subseteq S_3$. 
We consider cohomology by formally taking the alternating sum as an object in the Grothendieck group of $A_3 \times \Gal(\bar \bF_q/\bF_q)$-representations with $\ov{\bQ}_\ell$-coefficients: for any variety over $\bF_q$ with $A_3$-action we set
\[
\rH^\ast_{\rc}(-) := \sum_{i} (-1)^i \left[\rH^i_{\rc}(-_{\bar \bF_q},\ov{\bQ}_\ell)\right]
\]
as a virtual $A_3 \times \Gal(\bar \bF_q/\bF_q)$-representation. Here we denote by $[-]$ the class  of a representation in the Grothendieck group. Recall that the cohomology of $\bA^1$ agrees with  the inverse of the Tate twist
\[
\rH^\ast_{\rc}(\bA^1) = \left[\rH^2_{\rc}(\bA^1_{\bar \bF_q},\ov{\bQ}_\ell) \right] =  \left[\ov{\bQ}_\ell(-1) \right],
\]
and, since the generator in degree $2$ is the fundamental class, moreover has trivial  $\tau$-action. 
\begin{lem}
\label{lem:cohU}
$\rH^\ast_{\rc}(U) =  \left[\ov{\bQ}_\ell(-1) \right] \cdot \big(\rH_{\rc}^\ast(Y) - \rH_{\rc}^\ast(S) - \rH_c^\ast(H \setminus T) \big)$.
\end{lem}
\begin{proof}
The constructible decomposition $V \cup  (H \setminus T) \cup T= Y$ shows
\[
\rH_{\rc}^\ast(V) = \rH_{\rc}^\ast(Y) - \rH_{\rc}^\ast(T) - \rH_c^\ast(H \setminus T).
\]
Since $T \simeq S$ by \eqref{eq:compareST}, we are left to prove
\[
\rH^\ast_{\rc}(U) =  \left[\ov{\bQ}_\ell(-1) \right] \cdot \rH^\ast_{\rc}(V).
\]
This holds more generally for any geometric line bundle and the argument is recalled for the convenience of the reader. The assertion is local with respect to a constructible decomposition of the base $V$. We may therefore assume that the bundle is trivial. Then the K\"unneth-formula  yields
\[
\rH^\ast_{\rc}(U) =  \rH_{\rc}^\ast(\bA^1 \times V) = \rH_{\rc}^\ast(\bA^1) \cdot \rH^\ast_{\rc}(V) 
= \left[\ov{\bQ}_\ell(-1) \right] \cdot \rH^\ast_{\rc}(V). \qedhere
\]
\end{proof}

We denote the $\mu_3$-invariants on cohomology for the antidiagonal $\mu_3$-action on $E \times E$ by
\[
\rH^\ast_{\rc}(E \times E)^{\mu_3} = \sum_{i=0}^4 (-1)^i \left[ \rH^i_{\rc}((E \times E)_{\bar \bF_q},\ov{\bQ}_\ell)^{\mu_3}\right],
\]
and similarly for $\mu_3$-stable and $\tau$-stable locally closed subvarieties defined over $\bF_q$.
Since the torsor action commutes (resp.\ is Galois conjugated) with the action by $A_3$ (resp.\ Galois action by $\Gal(\bar \bF_q/\bF_q)$), we obtain a well defined object in the Grothendieck group.

\begin{lem} 
\label{lem:cohY}
We have the following identities:
\begin{enumerate}
\item
\label{lemitem:cohYsep}
If  $p \not= 3$, then
\[
\rH_{\rc}^\ast(Y) = \rH^\ast_{\rc}(E \times E)^{\mu_3}.
\]
\item
\label{lemitem:cohYinsep}
If $p=3$, then 
\[
\rH_{\rc}^\ast(Y) = \rH^\ast_{\rc}(\bP^1 \times \bP^1).
\]
\end{enumerate}
\end{lem}
\begin{proof}
(1) Let $p \not= 3$. 
The $\mu_3$-torsor $Y' = E \times E  \to Y$ is finite \'etale outside the set of fixed points $T'_0$. (Note that having ramification only in codimension $2$ is no contradiction since $Y$ is not regular in the branch points of $j': Y' \to Y$.) Thus pull back identifies 
\[
\rH_{\rc}^\ast(Y \setminus T_0) = \rH^\ast_{\rc}(E \times E \setminus T'_0)^{\mu_3}.
\]
We therefore have 
\begin{align*}
\rH_{\rc}^\ast(Y) & = \rH_{\rc}^\ast(Y \setminus T_0) +  \rH_{\rc}^\ast(T_0)   = \rH^\ast_{\rc}(E \times E \setminus T'_0)^{\mu_3}   + \rH_{\rc}^\ast(T_0) \\
& = \rH^\ast_{\rc}(E \times E)^{\mu_3} - \rH^\ast_{\rc}(T'_0)^{\mu_3}  + \rH_{\rc}^\ast(T_0)   = \rH^\ast_{\rc}(E \times E)^{\mu_3}.
\end{align*}

(2) Let now $p =3$. The $\mu_3$-torsor $Y \to Z$ is now purely inseparable and the claim follows from topological invariance of \'etale cohomology.
\end{proof}

Let us abbreviate for all $\sigma \in S_3$ the smooth part of the divisor $H_\sigma$ as part of the divisor $H \cup T$ (this is actually an irreducible component, but we don't need that) by
\[
H_\sigma^0 := H_\sigma \setminus (T \cup H_0).
\]
Similarly $\Gamma_\sigma^0$ denotes the smooth part $\Gamma_\sigma \setminus (S \cup \Gamma_0)$ of $\Gamma_\sigma$ as a component of $S \cup \Gamma$.

\begin{lem}
\label{lem:cohHminusT}
We have the following identities:
\begin{enumerate}
\item
\label{lemitem:cohHminusTsep}
If  $p \not= 3$, then
\[
\rH_{\rc}^\ast(H \setminus T) = \sum_{\sigma \in S_3} \rH_{\rc}^\ast(H_\sigma^0) +
\sum_{n \geq 2} \rH_{\rc}^\ast(H_n).
\]
\item
\label{lemitem:cohHminusTinsep}
If $p=3$, then 
\[
\rH_{\rc}^\ast(H \setminus T) = \sum_{\sigma \in S_3} \rH_{\rc}^\ast(\Gamma^0_\sigma) +
\sum_{n \geq 2} \rH_{\rc}^\ast(\Gamma_n).
\]
\end{enumerate}
\end{lem}
\begin{proof}
(1) This is obvious from the constructible decomposition
\[
H \setminus T = \bigcup_{\sigma \in S_3} H_\sigma^0 \cup \bigcup_{n \geq 2} H_n.
\]

Assertion (2) follows by the same argument applied to $\Gamma \setminus S$ 
using that $\rH_{\rc}^\ast(H \setminus T) = \rH_{\rc}^\ast(\Gamma \setminus S)$ by topological invariance of \'etale cohomology.
\end{proof}

\subsection{Cohomology of the elliptic curve}
\label{sec:coh_ell_curve}
Let $p \not= 3$. 
The elliptic curve $E/\bF_q$ is the base change $E = \cE \otimes_{\bZ[1/3]} \otimes  \bF_q$ of the smooth integral model over $\Spec(\bZ[1/3])$
\[
\cE = \{Y(Y-Z)Z = X^3\} \subseteq \bP^2_{\bZ[1/3]}.
\]
The generic fibre $\cE_{\bQ}$ is an elliptic curve over $\bQ$ with CM by $\bZ[\zeta_3]$.
First, we consider cohomology of the geometric generic fibre as $\mu_3 \subseteq \Aut(\cE)$-module and  as $\Gal(\ov \bQ/\bQ)$-module. For this purpose we recall some CM-theory for the convenience of the reader. Complex uniformisation yields
\[
\cE(\bC) \simeq \bC/\bZ[\zeta_3],
\]
which follows from CM by $\bZ[\zeta_3]$ since $\bQ(\zeta_3)$ has class number $1$, and hence all lattices with an action by $\bZ[\zeta_3]$ are isomorphic to $\bZ[\zeta_3]$. This computes  
\[
\bZ[\zeta_3] \otimes \ov{\bQ}_\ell \simeq \rH^1(\cE_{\bar \bQ}, \ov{\bQ}_\ell) 
\]
as $\mu_3$-module. It follows that after fixing a faithful character  $\psi : \mu_3 \inj \ov{\bQ}_\ell^\times$
we have $1$-dimensional eigenspaces for $r= 1,2$ 
\[
\rH(\psi^r) \subset  \rH^1(\cE_{\bar \bQ}, \ov{\bQ}_\ell)
\]
on which $\mu_3$ acts by character $\psi^r$, and 
\[
\rH^1(\cE_{\bar \bQ}, \ov{\bQ}_\ell) =  \rH(\psi) \oplus \rH(\psi^2)
\]
Now $\sigma \in \Gal(\ov \bQ/\bQ)$ either fixes the eigenspaces $\rH(\psi^r)$ if $\sigma(\zeta_3) = \zeta_3$ or interchanges the summands if $\sigma(\zeta_3) = \zeta_3^{-1}$. This means that there are characters (these are the Hecke characters of Remark~\ref{rmk:heuristicerrorterm})
\[
\alpha, \beta :  \Gal(\ov{\bQ}/\bQ(\zeta_3)) \to \ov{\bQ}_\ell^\times 
\]
such that $\Gal(\ov{\bQ}/\bQ(\zeta_3))$ acts by $\alpha$ on $\rH(\psi)$ and by $\beta$ on $\rH(\psi^2)$. We rename the eigenspaces as
\[
\rH_\alpha = \rH(\psi), \quad \text{ and } \quad \ov{\rH}_\beta = \rH(\psi^2),
\]
with the bar indicating that on this summand $\mu_3$ acts via $\bar{\psi} = \psi^2$.

Note that if we take the inverse $\mu_3$-action, i.e., by precomposing with $\zeta \mapsto \zeta^{-1}$, then we have
\[
\rH^1(\cE_{\bar \bQ},\ov{\bQ}_\ell) = \ov{\rH}_\alpha \oplus \rH_\beta,
\]
according to the bar-convention for indicating the character by which $\mu_3$ acts.

\begin{lem}
\label{lem:cohEtimesEmu3}
Let $\mu_3$ act antidiagonally on $\cE \times \cE$. Then 
\[
\rH^i(\cE_{\bar{\bQ}} \times \cE_{\bar{\bQ}}, \ov{\bQ}_\ell)^{\mu_3} = \left\{
\begin{array}{ll}
\ov{\bQ}_\ell &  i = 0, \\[1ex]
0 & i = 1,3, \text{ and } i \geq 5, \\[1ex]
\ov{\bQ}_\ell(-1)  \oplus (H_\alpha \otimes \ov{\rH}_\alpha) \oplus (\ov{\rH}_\beta \otimes \rH_\beta) \oplus \ov{\bQ}_\ell(-1) & i = 2, \\[1ex] 
\ov{\bQ}_\ell(-2) & i = 4,\\
\end{array}
\right.
\]
as $\Gal(\ov{\bQ}/\bQ(\zeta_3))$-representation.
\end{lem}
\begin{proof}
This follows from the K\"unneth-formula together with the discussion above for $\rH^1$ and the well known and $\mu_3$-invariant $\rH^0(\cE_{\bar \bQ},\ov{\bQ}_\ell) = \ov{\bQ}_\ell$ and $\rH^2(\cE_{\bar \bQ},\ov{\bQ}_\ell) = \ov{\bQ}_\ell(-1)$.
\end{proof}

For $p \not= 3, \ell$, the Galois representation of the geometric generic fibre
\[
\rho: \Gal(\bar \bQ/\bQ) \to \GL(\rH^1(\cE_{\bar \bQ},\ov \bQ_\ell))
\]
is unramified in $p$, hence there is a well defined action of Frobenius $\rho(\Frob_p)$. Cospecialisation induces an isomorphism for all $i$ 
\[
\rH^i(\cE_{\bar \bQ},\ov \bQ_\ell) \simeq \rH^i(E_{\bar \bF_q},\ov \bQ_\ell)
\]
compatible with action of Frobenius by $\rho(\Frob_p)^e$ and $F$ where $q = p^e$. 

\begin{prop}
\label{prop:traceFEtimesEmu3}
Let $p \not= 3$, and let $\alpha, \beta$ be the eigenvalues of Frobenius $\rho(\Frob_p)$ on $\rH^1(E_{\bar \bF_q},\ov \bQ_\ell)$. Then we have for $q = p^e$
\[
\tr\big(\tau F|\rH^\ast_{\rc}(E \times E)^{\mu_3}\big) = (1+q)^2 + (\alpha^e + \beta^e)^2 - q(1+\qrB{-3}{q}).
\]
Moreover, for $q$ with $\qrB{-3}{q} = -1$ we have $\alpha^e + \beta^e = 0$.
\end{prop}
\begin{proof}
First, since $\tau$ acts on the abelian surface $E \times E$ by translation with $(P_1,0)$, it is part of a connected group of endomorphisms and, by homotopy invariance, it acts as identity on cohomology. We may therefore ignore $\tau$. 

\smallskip

If $\qrB{-3}{q} = -1$, then $F$ interchanges $\rH_\alpha$ and $\ov{\rH}_\beta$, and consequently
\[
\tr\big(F|(H_\alpha \otimes \ov{\rH}_\alpha) \oplus (\ov{\rH}_\beta \otimes \rH_\beta)\big) = 0.
\]
With $F$ also $\rho(\Frob_p)$ interchanges $\rH_\alpha$ and $\ov{\rH}_\beta$, hence its trace on $\rH^1(E_{\bar \bF_q},\ov \bQ_\ell)$ vanishes and $\alpha = -\beta$. Moreover $e$ must be odd, so that in this case
\[
\alpha^e + \beta^e = 0.
\]
The trace of $F$ on the remaining cohomology, see Lemma~\ref{lem:cohEtimesEmu3}, is easily computed as $(1+q)^2$ from which the claim follows. 

\smallskip

If $\qrB{-3}{q} = 1$, then $F$ preserves the two eigenspaces and acts by a matrix
\[
F \sim \matzz{\alpha^e}{}{}{\beta^e}
\]
on $\rH^1(E_{\bar \bF_q},\ov \bQ_\ell)$ with wlog eigenvalue $\alpha^e$ on $\rH_\alpha$ and eigenvalue $\beta$ on $\ov \rH_\beta$. It follows that 
\[
\tr\big(F|(H_\alpha \otimes \ov{\rH}_\alpha) \oplus (\ov{\rH}_\beta \otimes \rH_\beta)\big) = \alpha^{2e} + \beta^{2e}.
\]
Since $\alpha \beta = p$ we find using Lemma~\ref{lem:cohEtimesEmu3}
\[
\tr\big(\tau F|\rH^\ast_{\rc}(E \times E)^{\mu_3}\big) = (1+q)^2 + \alpha^{2e} + \beta^{2e} = (1+q)^2 + (\alpha^e + \beta^e)^2 - q(1+\qrB{-3}{q}). \qedhere
\]
\end{proof}

\begin{cor}
\label{cor:traceFEtimesEmu3}
Let $p \not= 3$, and let $q = p^e$. With $c_q$ as in Theorem~\ref{thm:main}
we have
\[
\tr\big(\tau F|\rH^\ast_{\rc}(E \times E)^{\mu_3}\big) = (1+q)^2 + q \cdot (c^2_q - 1 - \qrB{-3}{q}).
\]
Moreover, for $q$ with $\qrB{-3}{q} = -1$ we have $c_q = 0$.
\end{cor}
\begin{proof}
This follows immediately from Proposition~\ref{prop:traceFEtimesEmu3} and the definition of $c_q = (\alpha^e+\beta^e)/\sqrt{q}$.
\end{proof}

\subsection{Traces of Frobenius}
\label{sec:proof of main thm}

If $\tau$ permutes varieties defined over $\bF_q$, then cohomology becomes an induced module with respect to at least the $\tau$-action. But then $\tau F$ has trace $0$ on these parts. We will make use of this observation repeatedly in what follows.

\begin{lem}
\label{lem:traceFS}
$\tr(\tau F | \rH_{\rc}^\ast(S)) = 3(1+q)$.
\end{lem}
\begin{proof}
We compute
\begin{align*}
\tr(\tau F | \rH_{\rc}^\ast(S)) & = \tr\Big(\tau F | \sum_{i \in \{0,1,\infty\}} \rH_{\rc}^\ast(\bP^1 \times \{i\}) + \rH^\ast_{\rc}(\{i\} \times \bP^1)\Big) -  \tr(\tau F | \rH_{\rc}^\ast(S_0)) \\
& = 3 \cdot \tr(\tau F |  \rH_{\rc}^\ast(\bP^1)) + \tr(\tau F | \ind_{1}^{A_3} \rH_{\rc}^\ast(\bP^1)) - \tr(\tau F | \ind_1^{A_3} \rH_{\rc}^\ast(\{0,1,\infty\})) \\
& = 3 \cdot \tr(F |  \rH_{\rc}^\ast(\bP^1))  = 3(1+q),
\end{align*}
because $\tau$ acts trivially on cohomology of $\bP^1$.
\end{proof}

\begin{lem}
\label{lem:traceFHn}
We have the following traces of Frobenius:
\begin{enumerate}
\item
\label{lemitem:traceFHnsep}
If  $p \not= 3$ and $n \geq 2$, then
\[
\tr(\tau F | \rH_{\rc}^\ast(H_n)) = 0.
\]
\item
\label{lemitem:traceFHninsep}
If $p=3$, then 
\[
\tr(\tau F | \rH_{\rc}^\ast(\Gamma_n)) = \left\{
\begin{array}{ll}
0 &  n \not= 6, \\
1 & n=6.\\
\end{array}
\right.
\]
\end{enumerate}
\end{lem}
\begin{proof}
(1) By Lemma~\ref{lem:geometryoffixedpoints} we have only to consider $n=2$ and $n=3$. By the Lefschetz trace formula, these traces of Frobenius are the $\tau F$-fixed points of $H_n(\bar \bF_q)$. For $n=3$, Lemma~\ref{lem:H3} shows that there are no fixed points. For $n=2$ this is obvious because any fixed point of $H_n(\bar \bF_q)$ lies over a fixed point of $\Gamma_n(\bar \bF_q)$. The description of $\Gamma_2(\bar \bF_q)$ shows that there are none since $\tau$ acts on $Z$ as $\tau \times \id$ and thus permutes the entries $\{1/2,2,-1\}$ in the first coordinate cyclically, 
see Lemma~\ref{lem:actionoftauonZ}.

(2) 
By Lemma~\ref{lem:geometryoffixedpoints} we have only to consider $n=6$. Here there is a single point which is clearly $\tau F$-invariant.
\end{proof}

\begin{lem}
\label{lem:traceFHminusT}
We have the following traces of Frobenius:
\begin{enumerate}
\item
\label{lemitem:traceFHminusTsep}
If  $p \not= 3$ and $n \geq 2$, then
\[
\tr(\tau F | \rH_{\rc}^\ast(H \setminus T))  = 0.
\]
\item
\label{lemitem:traceFHminusTinsep}
If $p=3$, then 
\[
\tr(\tau F | \rH_{\rc}^\ast(H \setminus T)) = 1.
\]
\end{enumerate}
\end{lem}
\begin{proof}
The sum of cohomologies of $H_\sigma^0$ (resp. of $\Gamma_\sigma^0$) yields induced modules with respect to the $\tau$ action compatible with Frobenius, because $\tau$ permutes the components without fixed points, see Lemma~\ref{lem:tauonHGamma}. 

(1) We can compute by 
Lemma~\ref{lem:cohHminusT}\eqref{lemitem:cohHminusTsep} and 
Lemma~\ref{lem:traceFHn}\eqref{lemitem:traceFHnsep}
\[
\tr(\tau F | \rH_{\rc}^\ast(H \setminus T)) = \tr(\tau F | \sum_{\sigma} \rH_{\rc}^\ast(H_\sigma^0))  + \sum_{n \geq 2} \tr(\tau F | \rH_{\rc}^\ast(H_n))  = 0.
\]

(2) Analogously, we can compute by 
Lemma~\ref{lem:cohHminusT}\eqref{lemitem:cohHminusTinsep} and 
Lemma~\ref{lem:traceFHn}\eqref{lemitem:traceFHninsep}
\[
\tr(\tau F | \rH_{\rc}^\ast(H \setminus T)) = \tr(\tau F | \sum_{\sigma} \rH_{\rc}^\ast(\Gamma_\sigma^0))  + \sum_{n \geq 2} \tr(\tau F | \rH_{\rc}^\ast(\Gamma_n))  = 1. \qedhere
\]
\end{proof}

\begin{prop}
\label{prop:traceFU}
We have the following traces of Frobenius:
\begin{enumerate}
\item
\label{lemitem:traceFUsep}
If  $p \not= 3$, then
\[
\tr(\tau F | \rH_{\rc}^\ast(U))  = q\big(q^2  + q(c^2_q - 2 - \qrB{-3}{q})    -2 \big)  .
\]
\item
\label{lemitem:traceFUinsep}
If $p=3$, then 
\[
\tr(\tau F | \rH^\ast_{\rc}(U)) = q ( q^2 - q + 3).
\]
\end{enumerate}
\end{prop}
\begin{proof}
By  Lemma~\ref{lem:cohU} we have in both cases
\begin{align*}
\tr(\tau F | \rH_{\rc}^\ast(U)) & = \tr(\tau F | \rH_{\rc}^\ast(\bA^1)) \cdot \Big(
\tr(\tau F | \rH_{\rc}^\ast(Y)) - \tr(\tau F | \rH_{\rc}^\ast(H \setminus T)) - \tr(\tau F | \rH_{\rc}^\ast(S))
\Big) \\
& = q \cdot \Big(\tr(\tau F | \rH_{\rc}^\ast(Y)) - 3(1+q) - \left\{\begin{array}{ll}
0 & p \not= 3 \\
1 & p = 3
\end{array}
\right.
\Big).
\end{align*}
Here we also used Lemma~\ref{lem:traceFS} and Lemma~\ref{lem:traceFHminusT}.

(1) We compute further by Lemma~\ref{lem:cohY}\eqref{lemitem:cohYsep} and 
Corollary~\ref{cor:traceFEtimesEmu3} 
that
\begin{align*}
\tr(\tau F | \rH^\ast_{\rc}(U))  & = q\big(\tr(\tau F | \rH_{\rc}^\ast(E \times E)^{\mu_3})  - 3(1+q) \big)  \\
& = q\big( (1+q)^2  + q(c^2_q - 1 - \qrB{-3}{q})   - 3(1+q) \big)  \\
& = q\big(q^2  + q(c^2_q - 2 - \qrB{-3}{q})    -2 \big). 
\end{align*}

(2) We compute further by Lemma~\ref{lem:cohY}\eqref{lemitem:cohYinsep}, 
that
\[
\tr(\tau F | \rH^\ast_{\rc}(U))  = q\big((1+q)^2  - 3(1+q) - 1\big) = q(q^2-q-3). \qedhere
\]
\end{proof}

\begin{proof}[Proof of Theorem~\ref{thm:main}]
The formula for cubic twin prime polynomial pairs now follows immediately by combining 
Proposition~\ref{prop:LTF} with Proposition~\ref{prop:traceFU}, and by noting 
\[
\qrB{-3}{q} = \qrB{q}{3} = 
\begin{cases} 
1 & \text{ if } q \equiv 1 \pmod 3\\
-1 & \text{ if } q \equiv 2 \pmod 3,
\end{cases}
\]
by quadratic reciprocity, and $c_q = 0$ if $q \equiv 2 \pmod 3$ as in  
Corollary~\ref{cor:traceFEtimesEmu3}.
\end{proof}

\subsection{General scalar shift}
We now count prime polynomial pairs with difference $h \in \bF_q^\times$ beyond the case $h=1$. The corresponding parametrizing variety is the open $U_h  = X_h \cap \{ \Delta_x \cdot \Delta_y \cdot z \not= 0\}$ in the cubic twist 
\[
X_h = V\big(\sigma_2(\underline{x}) -\sigma_2(\underline{y}), \ \sigma_3(\underline{x}) - \sigma_3(\underline{y}) - h z)^3 \big) \subseteq \bP(M).
\]
of the variety $X = X_1$ parametrizing prime polynomial pairs of degree $3$ and shift $1$. In the following, we will decorate the notation with an index $h$ for the corresponding twisted geometric object.

The geometric description of $U_h$ is analogous to that for $U$. We have a geometric line bundle  $U_h \to V_h$, due to translation invariance. There is a compactification  $V_h \subseteq Y_h$ and a $\mu_3$-torsor $j_h : Y_h \to Z_h$ given in bihomogeneous coordinates $[u:v],[r:s]$ for $Z_h = Z = \bP^1 \times \bP^1$ by the equation 
\[
h z^3  = uv(u-v) \cdot rs (s-r).
\]
The branch locus of the torsor $j_h$ is still  $S_h = S$.

Let $E/\bF_q$ be the elliptic curve $\{w^3 = uv(u-v)\}$ as before, and let $h t^3 = rs(r-s)$ be the equation of its cubic twist $E_h/\bF_q$. The arithmetic of the $\mu_3$-cover $j':E \times E \to Y$ necessary to compute cohomology (as virtual $A_3 \times \Gal_{\bF_q}$-representation) is twisted to the torsor $j'_h : Y'_h = E \times E_h \to Y_h$ given by 
\[
j_h'([u:v:w], [r:s:t]) =  [ur:us:vr:vs:-wt].
\]
Let $\alpha$, $\beta$ be the eigenvalues of $p$-Frobenius on $\rH^1(E_{\bar \bF_q},\bQ_\ell)$, and let $q =p^e$ as usual. Then there is a cube root of unity $\zeta$, with $\zeta = 1$ if and only if $h$ is a cube in $\bF_q^\times$, such that 
\[
\zeta \alpha^e, \ q \cdot (\zeta \alpha^e)^{-1} = \zeta^2 \beta^e
\]
are the eigenvalues of $q$-Frobenius on $\rH^1(E_{h,\bar \bF_q},\bQ_\ell)$. The Lefschetz trace formula shows
\[
\zeta \alpha^e+  \zeta^2 \beta^e = 1+q - \#E_h(\bF_q) =: \sqrt{q} \cdot c_{q,h},
\]
with $c_q = c_{q,h}$ if $h$ is a cube in $\bF_q^\times$. We get that 
\begin{equation}
\label{eq:twistedtraceFrobSquared}
\frac{\zeta \alpha^{2e} +  \zeta^2 \beta^{2e}}{q} =  c_{q,h} \cdot c_q - (\zeta + \zeta^2) = 
\begin{cases}
c_q^2 - 2 &  \text{ if $h$ is a cube, hence $\zeta = 1$, } \\[1ex]
c_{q,h} \cdot c_q  +  1 & \text{ if $h$ is not a cube, hence $\zeta \not= 1$. } 
\end{cases}
\end{equation}

For $q = p^e$ not divisible by $3$, we will need the trace $\tr(\tau F | \rH^\ast_{\rc}(E \times E_h)^{\mu_3})$ which we compute as in 
Proposition~\ref{prop:traceFEtimesEmu3}.
If $\qrB{-3}{q}  =-1$ and with the notation of loc.~cit.\ used analogously, we have 
\[
\tr\big(F|(H_\alpha \otimes \ov{\rH}_\alpha) \oplus (\ov{\rH}_\beta \otimes \rH_\beta)\big) = 0.
\]
If $\qrB{-3}{q}  =1$, we have to take into account that the second  factor of $E \times E_h$ is twisted, hence the eigenvalues of Frobenius are multiplied by cubic roots of unity as explained above. We obtain 
\[
\tr\big(F|(H_\alpha \otimes \ov{\rH}_\alpha) \oplus (\ov{\rH}_\beta \otimes \rH_\beta)\big) = \zeta \alpha^{2e} +  \zeta^2 \beta^{2e}.
\]
Consequently, the analogue of Corollary~\ref{cor:traceFEtimesEmu3} based on \eqref{eq:twistedtraceFrobSquared} is 
\begin{align}
\tr\big(\tau F | \rH^\ast_{\rc}(E_h \times E)^{\mu_3}\big) 
& = (1+q)^2 +  \tr\big(F|(H_\alpha \otimes \ov{\rH}_\alpha) \oplus (\ov{\rH}_\beta \otimes \rH_\beta)\big) \notag \\
& =
\begin{cases}
(1+q)^2 & \text{ if } \qrB{-3}{q} = -1, \\[1ex]
(1+q)^2 + q \cdot (c_{q}^2   - 2)& \text{ if } \qrB{-3}{q} = 1 \text{ and $h$ is a cube, } \\[1ex]
(1+q)^2 + q \cdot ( c_{q,h} \cdot c_q  + 1) & \text{ if } \qrB{-3}{q} = 1 \text{ and $h$ is not a cube. } 
\end{cases}
\label{eq:twistedHExE}
\end{align}

For the point counting we also need $\tr\big(\tau F | \rH_{\rc}^\ast(H_{3,h})\big)$ which is best computed, as in Lemma~\ref{lem:H3}, by counting the $\tau F$-fixed points of 
\[
H_{3,h}(\bar \bF_q) \simeq \{\xi \ ; \ h\xi^3 = -1\} \times  \{-\zeta_3,-\zeta_3^2\} \times \{-\zeta_3,-\zeta_3^2\}.
\]
Here $F$ acts as $q$-Frobenius, and $\tau$ fixes the second and third component, while in the first component we have (for $\lambda = -\zeta_3, -\zeta_3^2$ the factor $1/(\lambda -1)$ equals $-\lambda \in \mu_3$)
\[
\tau^\ast(\xi) = \xi \frac{1}{\lambda - 1} = - \lambda \xi.
\]
Therefore $(\xi,\lambda,\mu) \in H_{3,h}(\bar \bF_q)$ is fixed by $\tau F$  if and only if
\[
(\xi,\lambda,\mu) = \tau F(\xi, \lambda,\mu) = \tau(\xi^q,\lambda^q, \mu^q) = (-\lambda^q \xi^q, \lambda^q, \mu^q).
\]
If $\tau F$-fixed points in $H_{3,h}(\bar \bF_q)$ exist, we must have $\lambda^q = \lambda$, i.e., $\zeta_3$ is fixed under $q$-Frobenius. That means $\zeta_3 \in \bF_q^\times$ and furthermore $1 =  \qrB{-3}{q}$. In this case, the tuple $(\xi,\lambda,\mu)$ is a fixed point if and only if
\[
\xi/F(\xi) = - \lambda.
\]
Since $h\xi^3 = -1$, this happens if and only if $h$ is not a cube in $\bF^\times_q$. In that case we have three values of $\xi$, each determines a unique suitable value for $\lambda$, and $\mu$ can be arbitrary in $\{-\zeta_3,-\zeta_3^2\}$. This shows:
\begin{equation}
\label{eq:twistedH3}
\tr\big(\tau F | \rH_{\rc}^\ast(H_{3,h})\big) = 
\begin{cases}
0 & \text{ if } \qrB{-3}{q} = -1, \\[1ex]
0 & \text{ if } \qrB{-3}{q} = 1 \text{ and $h$ is a cube, } \\[1ex]
6 & \text{ if } \qrB{-3}{q} = 1 \text{ and $h$ is not a cube. } 
\end{cases}
\end{equation}

\begin{proof}[Proof of Theorem~\ref{thm:main_twisted}]
For $3 \mid q$, twisting $U$ to $U_h$ has no effect on the point count because the relevant $\mu_3$-torsor is purely inseparable. So in this case the formula follows from 
Theorem~\ref{thm:main}. 

For $3 \nmid q$, we compute as in the proof of Theorem~\ref{thm:main} and of Proposition~\ref{prop:traceFU} based on \eqref{eq:twistedHExE} and \eqref{eq:twistedH3}
\begin{align*}
\pi(3,q; (0,h)) & = \frac{1}{9} \tr(\tau F | \rH_{\rc}^\ast(U_h)) =  \frac{q}{9}  \cdot \Big(
\tr(\tau F | \rH_{\rc}^\ast(Y_h)) - \tr(\tau F | \rH_{\rc}^\ast(H_h \setminus T_h)) - \tr(\tau F | \rH_{\rc}^\ast(S))
\Big) \\
& = \frac{q}{9}  \cdot \Big(\tr(\tau F | \rH^\ast_{\rc}(E_h \times E)^{\mu_3}) -  \tr(\tau F | \rH_{\rc}^\ast(H_{3,h}))  -  3(1+q) 
\Big) \\
& = \frac{q}{9} \cdot 
\begin{cases}
q^2 - q - 2  & \text{ if } \qrB{-3}{q} = -1, \\[1ex]
q^2  + q \cdot (c_q^2  - 3) - 2  & \text{ if } \qrB{-3}{q} = 1 \text{ and $h$ is a cube, } \\[1ex]
q^2  + q \cdot c_{q,h} \cdot c_q    -  8 & \text{ if } \qrB{-3}{q} = 1 \text{ and $h$ is not a cube. } 
\end{cases} \qedhere
\end{align*}
\end{proof}



\begin{thebibliography}{WaMi71}

\bibitem[ABR15]{ABSR}
J.\thinspace{}C.~Andrade, L.~Bary-Soroker, and Z. Rudnick, 
\emph{Shifted convolution and the Titchmarsh divisor problem over $F_q[t]$},
Philos. Trans. A 373 (2015), no. 2040, 20140308, 18 pp. 

\bibitem[BaB15]{BBS}
E. Bank and L. Bary-Soroker, 
\emph{Prime polynomial values of linear functions in short intervals},
J. Number Theory \textbf{151} (2015), 263-275. 

\bibitem[Bar12]{BarySoroker}
L. Bary-Soroker,
\emph{Hardy-Littlewood tuple conjecture over large finite fields},
Int.\ Math.\ Res.\ Not.\ (2014), no.\ 2, 568--575.

\bibitem[BSF18]{BSF}
L. Bary-Soroker and A. Fehm, 
\emph{Correlations of sums of two squares and other arithmetic functions in function fields}, preprint 2017, 
\href{https://arxiv.org/abs/0912.1702}{arXiv:1701.04092}, to appear.

\bibitem[BeP09]{Bender}
A.\thinspace{}O. Bender, P. Pollack,
\emph{On quantitative analogues of the Goldbach and twin prime conjectures over $F_q[t]$}, 
preprint 2009, \href{https://arxiv.org/abs/0912.1702}{arXiv:0912.1702}.

\bibitem[Bru19]{Brun}
V. Brun, 
\emph{La s\'erie $1/5+1/7+1/11+1/13+1/17+1/19+1/29+1/31+1/41+1/43+1/59+1/61+\cdots,$ o\`u les d\'enominateurs sont nombres premiers jumeaux est convergente ou finie}, 
Bulletin des Sciences Mathematiques. 43: 100--104, 124--128, (1919).

\bibitem[Car15]{Carmon}
D. Carmon,
\emph{The autocorrelation of the M\"obius function and Chowla's conjecture for the rational function field in characteristic 2},
Phil. Trans. R. Soc. A (2015) 373 20140315. 

\bibitem[Cas15]{Castillo}
A. Castillo, Ch. Hall, R.\thinspace{}J. Lemke Oliver, P. Pollack, and L. Thompson,
\emph{Bounded gaps between primes in number fields and function fields},
Proceedings of the AMS \text{143} (2015),  no.\ 7, 2841--2856. 

\bibitem[Che66]{Chen}
J.\thinspace{}R. Chen,
\emph{On the representation of a large even integer as the sum of a prime and the product of at most two primes},
Kexue Tongbao \textbf{11} (1966), no.\ 9, 385--386.

\bibitem[Deu53]{deuring}
M.~Deuring, 
\emph{Die Zetafunktion einer algebraischen Kurve vom Geschlechte Eins I-III}, 
Nachr.\ Akad.\ Wiss.\ G\"ottingen, Math.-Phys.\ Kl.\ \textbf{1953} (1953), 85--94, \textbf{1955} (1955), 13--42, \textbf{1956} (1956), 37--76.

\bibitem[Ent14]{Entin}
A. Entin, 
\emph{On the Bateman-Horn conjecture for polynomials over large finite fields}, 
Compositio Mathematica \textbf{152} (2016), no.\ 12, 2525--2544.

\bibitem[FK88]{freitagkiehl}
E.~Freitag and R.~Kiehl,
\emph{\'Etale cohomology and the Weil conjecture},
Springer 1988, xviii+320 pp.

\bibitem[GS18]{GorodetskySawin}
O. Gorodetsky and W. Sawin, 
work in progress.

\bibitem[Gra15]{Granville}
A. Granville, 
\emph{Primes in intervals of bounded length},
Bulletin of the AMS \textbf{52} (2015), 171--222.

\bibitem[GrT08]{GreenTao}
B. Green and T. Tao,
\emph{The primes contain arbitrarily long arithmetic progressions}, 
Annals of Mathematics \textbf{167}  (2008), no.\ 2, 481--547. 

\bibitem[GTZ12]{GTZ}
B. Green, T. Tao, and T. Ziegler,
\emph{An inverse theorem for the Gowers $U^{s+1}[N]$-norm},
Annals of Mathematics \textbf{176} (2012), no.\ 2, 1231--1372.

\bibitem[Hal06]{Hall}
Ch. Hall, 
\emph{L-functions of twisted Legendre curves}, 
Journal of Number Theory \textbf{119} (2006), no.\ 1, 128--147. 

\bibitem[HL23]{HardyLittlewood}
G.\thinspace{} H. Hardy, J.\thinspace{}E. Littlewood,
\emph{Some problems of `Partitio numerorum'; III: On the expression of a number as a sum of primes},
Acta Math. \textbf{44} (1923), no.\ 1, 1--70.

\bibitem[HM17]{HastMatei}
D.~Hast and V.~Matei, \emph{Higher moments of arithmetic functions in Short Intervals}, IMRN, in print. 

\bibitem[Hec18]{Hecke}
E.~Hecke, 
\emph{Eine neue Art von Zetafunktionen und ihre Beziehungen zur Verteilung der Primzahlen}, 
Math.\ Z. \textbf{1} (1918), no.\ 4, 357--376;
\emph{Eine neue Art von Zetafunktionen und ihre Beziehungen zur Verteilung der Primzahlen},
Math.\ Z. \textbf{6} (1920), no.\ 1-2, 11--51.

\bibitem[Kat12a]{Katz1}
N.\thinspace{}M. Katz,
\emph{On a question of Keating and Rudnick about primitive Dirichlet characters with squarefree conductor},
Int.\ Math.\ Res.\ Not.\ (2013), no. 14, 3221--3249.

\bibitem[Kat12b]{Katz2}
N.\thinspace{}M. Katz,
\emph{Witt vectors and a question of Keating and Rudnick},
Int.\ Math.\ Res.\ Not.\  (2013), no.\ 16, 3613--3638.

\bibitem[KeR14]{KeatingRudnick}
J.\thinspace{}P. Keating and Z. Rudnick,
\emph{The variance of the number of prime polynomials in short intervals and in residue classes},
Int.\ Math.\ Res.\ Not.\ (2014), no. 1, 259--288.

\bibitem[KRG16]{KeatingRodittyGershon}
J.\thinspace{}P. Keating, and E. Roditty-Gershon, 
\emph{Arithmetic Correlations Over Large Finite Fields},
Int.\ Math.\ Res.\ Not.\ (2016), 860--874.

\bibitem[LMFDB]{LMFDB}
The LMFDB Collaboration, 
\emph{The $L$-functions and Modular Forms Database}, {\tt http://www.lmfdb.org}, 2017, [Online; accessed 28 September 2017].



\bibitem[May15]{Maynard}
J. Maynard, 
\emph{Small Gaps Between Primes},
Annals of Mathematics (2), \textbf{181} (2015), 383--413.
 
\bibitem[Pol08]{Pollack}
P. Pollack, 
\emph{An explicit approach to Hypothesis H for polynomials over finite fields}, in:
Anatomy of integers. Proceedings of a conference on the anatomy of integers, Montreal, March 13th-17th, 2006, J.-M. De Koninck, A. Granville, F. Luca, eds., CRM Proceedings and Lecture Notes, vol. 46, 2008, pp. 47--64.
 
\bibitem[Pol14]{Polymath}
D.\thinspace{}H.\thinspace{}J.~Polymath,
\emph{New equidistribution estimates of Zhang type},
Algebra \& Number Theory, \textbf{8} (2014), no.\ 9, 2067--2199.

\bibitem[Ser73]{serre-coursearithmetic}
J.-P.~Serre,
\emph{A Course in Arithmetic}, 
Springer GTM \textbf{7}, Springer 1973, viii+115 pp.

\bibitem[Sil09]{silverman}
J.\thinspace{}H.~Silverman,
The arithmetic of ellitpic curves, 
Springer GTM \textbf{106}, $2$nd edition, Springer 2009, xx+513 pp.

\bibitem[Sk01]{sko:torsors}  
A.~Skorobogatov,
\emph{Torsors and rational points}, 
Cambridge Tracts in Mathematics \textbf{144}, Cambridge University Press 2001, viii+187 pp.

\bibitem[Su17]{sutherland}
A.\thinspace{}V.~Sutherland,
\emph{Sato-Tate distributions},
preprint 2017,
\href{https://arxiv.org/abs/1604.01256v4}{arXiv:1604.01256v4}, 45 pp.

\bibitem[Zha14]{Zhang}
Y. Zhang,
\emph{Bounded gaps between primes},
Annals of Mathematics \textbf{179} (2014), no.\ 3, 1121--1174. 

\end{thebibliography}
\end{document}